\documentclass[12pt,fleqn]{article}
\usepackage{mathtools,amsfonts,amsmath, amsthm,amssymb}
\usepackage{accents}
\usepackage{chngcntr}
\counterwithout{figure}{section}

\def\lab{\label}
\usepackage[active]{srcltx}

\usepackage{mathrsfs} 
\usepackage{csquotes}
\usepackage{latexsym}
\usepackage{esint}

\usepackage[margin=22mm]{geometry} 
 \numberwithin{equation}{section}

\usepackage{color}
 \definecolor{db}{rgb}{0.0,0.0,0.8} 
\definecolor{dg}{rgb}{0.0,0.55,0.14}
\definecolor{dr}{rgb}{0.5,0,0.07}

  \usepackage{hyperref}
 \hypersetup{frenchlinks=true,colorlinks=true,linkcolor=db,citecolor=dg,
 bookmarks=true}

 \usepackage{enumerate}
\usepackage{authblk}

\newtheorem{theorem}{Theorem}[section]
\newtheorem{proposition}[theorem]{Proposition}
\newtheorem*{theorem*}{Theorem}

\newtheorem*{claim*}{Claim}
\newtheorem{lemma}[theorem]{Lemma}
\newtheorem{corollary}[theorem]{Corollary}
\theoremstyle{definition}

\theoremstyle{definition}

\theoremstyle{definition}

\theoremstyle{definition}

\theoremstyle{definition}

\theoremstyle{definition}
\newtheorem{remark}[theorem]{Remark}
\theoremstyle{definition}

\newtheorem{open-problem}{Open Problem}
\newcounter{step}

\newcommand{\rlemma}[1]{Lemma~\ref{#1}}
\newcommand{\rth}[1]{Theorem~\ref{#1}}
\newcommand{\rcor}[1]{Corollary~\ref{#1}}

\newcommand{\rprop}[1]{Proposition~\ref{#1}}

\def\be{\begin{equation}}
\def\ee{\end{equation}}
\def\bes{\begin{equation*}}
\def\ees{\end{equation*}}
\def\bt{\begin{theorem}}
\def\et{\end{theorem}}
\def\bpr{\begin{proposition}}
\def\epr{\end{proposition}}
\def\bl{\begin{lemma}}
\def\el{\end{lemma}}
\def\bc{\begin{corollary}}
\def\ec{\end{corollary}}
\def\br{\begin{remark}}
\def\er{\end{remark}}
\def\ben{\begin{enumerate}}
\def\bena{\begin{enumerate}[a)]}
\def\een{\end{enumerate}}
\def\bit{\begin{itemize}}
\def\iit{\end{itemize}}

\def\dist{\operatorname{dist}}
\def\Dist{\operatorname{Dist}}

\def\curl{\operatorname{curl}}
\def\div{\operatorname{div}}
\def\deg{\operatorname{deg}}

 \newcommand{\Prod}{\mathop{\prod}\limits}
\DeclareMathAlphabet{\mathonebb}{U}{bbold}{m}{n}

\def\R{{\mathbb R}}

\def\C{{\mathbb C}}
\def\Z{{\mathbb Z}}

\usepackage{csquotes}

\def\fo{\forall\, }

\def\va{\varphi}

\def\d{\displaystyle}
\def\im{\imath}
\def\ve{\varepsilon}

\def\na{\nabla}

\def\so{{\mathbb S}^1}
\def\st{{\mathbb S}^2}

\def\woo{W^{1,1}}
\def\wop{W^{1,p}}

\renewcommand{\vec}[1]{\boldsymbol{#1}}
\newcommand{\wtu}{{\widetilde u}}
\newcommand{\wtw}{{\widetilde w}}
\newcommand*{\dt}[1]{%
  \accentset{\mbox{\large\bfseries .}}{#1}}
\newcommand\blfootnote[1]{%
  \begingroup
  \renewcommand\thefootnote{}\footnote{#1}%
  \addtocounter{footnote}{-1}%
  \endgroup
}

\title{Distances between classes in $W^{1,1}(\Omega;{\mathbb S}^1)$}
\author[1,3]{Haim Brezis}
\author[2]{Petru Mironescu}
\author[3]{Itai Shafrir}
\affil[1]{Department of Mathematics, Rutgers
    University,  USA}
\affil[2]{Universit\'e de Lyon,  Universit\'e Lyon 1, CNRS UMR 5208,  Institut Camille Jordan,  69622 Villeurbanne, France}
\affil[3]{Department of Mathematics, Technion - I.I.T., 32 000 Haifa, Israel}

\begin{document}

\maketitle
\begin{abstract}
We introduce an equivalence relation on the space
$W^{1,1}(\Omega;{\mathbb S}^1)$ which classifies maps 
according to their
\enquote{topological singularities}.
We establish sharp bounds for the distances (in the usual sense and in
the Hausdorff sense) between the equivalence classes. Similar
questions are examined 
for the space $W^{1,p}(\Omega;{\mathbb S}^1)$ when $p>1$. 
\end{abstract}

\section{Introduction}
\label{sec:Omega}
 Let $\Omega$ be a smooth bounded domain in
 $\R^N$, $N\geq2$. (Many of the results in this paper remain valid if
 $\Omega$ is replaced by a manifold $\mathcal{M}$, with or without
 boundary, and the case $\mathcal{M}=\so$ is already of interest
 (see \cite{bms2,bmsCRAS}).) In
 some places we will assume in addition that $\Omega$ is simply
 connected (and this will be mentioned explicitly). Our basic setting
 is 
\bes
\woo(\Omega;\so)=\{u\in\woo(\Omega;\R^2)\simeq
\woo(\Omega;\C);\,|u|=1\text{ a.e.}\}.
\ees
It is clear that if $u,v\in\woo(\Omega;\so)$ then
$uv\in\woo(\Omega;\so)$; moreover
\begin{equation}
  \label{eq:X.1}
\text{if $u_n\to u$ and $v_n\to v$ in $\woo(\Omega;\so)$ then
  $u_nv_n\to uv$ in $\woo(\Omega;\so)$.}
\end{equation}
In particular, $\woo(\Omega;\so)$ is a topological group.
 We call the attention of the reader that maps $u$ of the form $u=e^{\im\varphi}$
with $\varphi\in\woo(\Omega;\R)$ belong to $\woo(\Omega;\so)$. However
they do not exhaust $\woo(\Omega;\so)$: there exist maps in
$\woo(\Omega;\so)$ which {\it cannot} be written as $u=e^{\im\varphi}$
for some $\varphi\in\woo(\Omega;\R)$. A typical example is the map
$u(x)=x/|x|$ in $\Omega=$unit disc in $\R^2$; This was originally
observed in \cite{bethuelzheng} (with roots in \cite{su}) and is based
on degree theory; see also \cite{bbm-lifting,bm}. 
Set 
\begin{equation}
  \label{eq:E}
  {\mathcal E}=\{u\in W^{1,1}(\Omega;{\mathbb S}^1);\,u=e^{\im \varphi}\text{ for some }
 \varphi\in W^{1,1}(\Omega;\R)\}.
\end{equation}

We claim that $\mathcal{E}$ is closed in $W^{1,1}(\Omega;{\mathbb
  S}^1)$. Indeed, let $u_n=e^{\im\varphi_n}$ with $u_n\to u$ in
$W^{1,1}$.
Then $\nabla\varphi_n=-\im{\overline u}_n\nabla u_n$ converges in $L^1$ to $-\im{\overline u}\nabla u$. 
By adding an integer multiple of $2\pi$ to $\varphi_n$ we may assume
that $\big|\int_\Omega\varphi_n\big|\leq 2\pi|\Omega|$. Thus,
a subsequence of  $\{\varphi_n\}$ converges in $W^{1,1}$ to some
$\varphi$ and $u=e^{\im\varphi}$.

Clearly
\begin{equation}
  \label{eq:X1}
  \mathcal{E}\subset \overline{C^\infty(\overline{\Omega};\so)}^{\woo}.
\end{equation}
 Indeed, if $u\in\mathcal{E}$, write $u=e^{\im\varphi}$ for some
 $\varphi\in W^{1,1}(\Omega;\R)$; let $\varphi_n\in
 C^\infty(\overline{\Omega};\R)$ be such that $\varphi_n\to\varphi$ in
 $W^{1,1}$. Then, $u_n=e^{\im\varphi_n}\in
 C^\infty(\overline{\Omega};\so)$ and converges to $u$ in $W^{1,1}$.
However equality in \eqref{eq:X1} fails in general. For example when
$\Omega=\{x\in\R^2;\,1<|x|<2\}$, the map $u(x)=x/|x|$ is smooth, but
$u\notin \mathcal{E}$; as above the
nonexistence  of $\varphi$ is an easy consequence of degree theory.
On the other hand, if $\Omega$ is simply connected, equality in
\eqref{eq:X1} does hold since
$C^\infty(\overline{\Omega};\so)\subset\mathcal{E}$ (recall that any
$u\in C^\infty(\overline{\Omega};\so)$ can be written as
$u=e^{\im\varphi}$ with $\varphi\in C^\infty(\overline{\Omega};\R)$) and
$\mathcal{E}$ is closed in $\woo(\Omega;\so)$).

To each $u\in\woo(\Omega;\so)$ we associate a number $\Sigma(u)\geq0$
 defined by 
\begin{equation}
  \label{eq:1*}
  \Sigma(u)=\inf_{\varphi\in W^{1,1}(\Omega;\R)} \int_\Omega |\nabla
  (ue^{-\im\varphi})|=\inf_{v\in\mathcal{E}} \int_\Omega |\nabla
  (u{\overline v})|.
\end{equation}
As explained in Section~\ref{sec:comments} the quantity $\Sigma(u)$
plays an extremely important role in many questions involving
$\woo(\Omega;\so)$; it has also an interesting geometric
interpretation. Note that 
\begin{equation}
\label{eq:120}
 u\in\mathcal{E} \Longleftrightarrow \Sigma(u)=0.
\end{equation}
The implication $\Longrightarrow$ is clear. For the reverse
implication, assume that $\Sigma(u)=0$, i.e., there exists a sequence
$v_n\in\mathcal{E}$ such that $\int_\Omega|\nabla(u\overline{v}_n)|\to
0$. Then (modulo a subsequence) $u\overline{v}_n\to C$ in $\woo$,  for some
constant $C\in\so$; therefore $v_n\to \overline{C}u$ in $\woo$ and thus
$u\in\mathcal{E}$ (since $\mathcal{E}$ is closed). 
\par 
 In some sense $\Sigma(u)$ measures how much a general
 $u\in\woo(\Omega;\so)$ ``deviates'' from $\mathcal{E}$. More
 precisely we will prove that 
 \begin{equation}
\label{eq:118}
  \frac{2}{\pi}\, \Sigma(u)\leq
   \inf_{v\in\mathcal{E}}\int_\Omega|\nabla(u-v)|\leq \Sigma(u),
 \end{equation}
 with optimal constants. This will be derived as a very special case
 of our main result \rth{th:main}.
 In order to state it we need to describe
 a decomposition of the space $\woo(\Omega;\so)$ according to the
following  equivalence relation in 
 $\woo(\Omega;\so)$:
 \begin{equation}
   \label{eq:equiv}
  u\sim v \text{ if and only if }u=e^{\im \varphi}\, v\text{ for some }
 \varphi\in W^{1,1}(\Omega;\R);
 \end{equation}
in other words, $u\sim v$ if and only if $\Sigma(u\overline{v})=0$.
We denote
 by ${\mathcal E}(u)$ the equivalence class of an element $u\in
 W^{1,1}(\Omega;{\mathbb S}^1)$, that is
 \begin{equation*}
   {\mathcal E}(u)=\{ue^{-\im\varphi};\,\varphi\in W^{1,1}(\Omega;\R)\}.
 \end{equation*}
In particular,
 $\mathcal{E}(1)=\mathcal{E}$.
It is easy to see that for every $u\in\woo(\Omega;\so)$,
$\mathcal{E}(u)$ is closed (it suffices to apply \eqref{eq:X.1} and
the fact that $\mathcal{E}$ is closed). In Section~\ref{sec:comments} we
will give an interpretation of the equivalence relation $u\sim v$ in
terms of the ``topological singularities'' of $u$ and $v$. We may rewrite \eqref{eq:1*} as 
\begin{equation}
  \label{eq:1**}
  \Sigma(u)=\inf_{v\in\mathcal{E}(u)} \int_\Omega |\nabla v|.
\end{equation}

Given $u_0,v_0\in W^{1,1}(\Omega;{\mathbb S}^1)$ the following
quantities will play a crucial role throughout the paper:
\begin{align}
\label{eq:81}
 d_{\woo}(u_0,\mathcal{E}(v_0))&:=\inf_{v\sim v_0} \int_\Omega
  |\nabla (u_0-v)|,\\
\label{eq:10}
   \dist_{W^{1,1}} ({\cal E}(u_0),{\cal E}(v_0))&:=\inf_{u\sim u_0} d_{\woo}(u, {\cal E}({v_0}))
   =\inf_{u\sim u_0}\, \inf_{v\sim v_0} \int_\Omega |\nabla(u-v)|,
   \\
   \label{eq:11}
   \Dist_{W^{1,1}}({\cal E}(u_0),{\cal E}(v_0))&:=\sup_{u\sim u_0} d_{\woo}(u, {\cal E}({v_0}))
   =\sup_{u\sim u_0}\, \inf_{v\sim v_0} \int_\Omega |\nabla(u-v)|,
     \end{align}
so that $\dist_{W^{1,1}} ({\cal E}(u_0),{\cal E}(v_0))$ is precisely
the distance between the classes $\mathcal{E}(u_0)$ and
$\mathcal{E}(v_0)$. On the other hand we will see below, as a
consequence of 
\eqref{eq:19b}, that
$\Dist_{\woo}$ is symmetric, a fact which is not clear from its
definition. This implies that $\Dist_{W^{1,1}}$ coincides with the Hausdorff distance 
 \bes
  H-\dist_{W^{1,1}}({\cal E}(u_0),{\cal E}(v_0)):=\max\left(\Dist_{W^{1,1}}({{\cal E}(u_0)}, {{\cal E}(v_0)}),\, \Dist_{W^{1,1}}({{\cal E}(v_0)}, {{\cal E}(u_0)})\right)
 \ees
between ${{\cal E}(u_0)}$ and ${{\cal E}(v_0)}$.
Our  main result is
 \begin{theorem}
   \label{th:main}
 For every  $u_0,v_0\in W^{1,1}(\Omega;{\mathbb S}^1)$ we have
 \begin{equation}
   \label{eq:19a}
 \dist_{W^{1,1}} ({\cal E}(u_0),{\cal E}(v_0))=
 \frac{2}{\pi}\Sigma(u_0{\overline v}_0)
\end{equation}
and
\be
\lab{eq:19b}
 \Dist_{W^{1,1}} ({\cal E}(u_0),{\cal E}(v_0))= \Sigma(u_0{\overline v}_0).
 \ee
 \end{theorem}
 The two assertions in \rth{th:main} look very simple but the proofs  are
 quite tricky; they are presented in Sections~\ref{sec:lb} and
 \ref{sec:19b}.
\par A useful device for constructing maps in the same equivalence class is
the following (see \rlemma{lem:T} below).  Let
     $T\in\text{Lip}(\so;\so)$ be a map of degree one. Then
     \begin{equation}
       \label{eq:X2}
   T\circ u\sim u,\  \forall\, u\in\woo(\Omega;\so).
     \end{equation}
 It turns out that this simple device plays a very significant
 role in the proofs of most of our main results. It allows us to work on
 the target space only,  thus avoiding difficulties due to the possibly
 complicated geometry and/or topology of the domain (or manifold)
 $\Omega$. A first example of an application of this technique is given
 by the proof of the following version of the \enquote{dipole
 construction}; it is the main ingredient in the proof of
 inequality \enquote{$\leq$} in \eqref{eq:19b}.
 \begin{proposition}(H.~Brezis and P.~Mironescu~\cite[Proposition~2.1]{bm})
   \label{prop:dipole}
Let $u\in\woo(\Omega;{\mathbb S}^1)$. Then there exists a sequence $\{u_n\}\subset\mathcal{E}(u)$ satisfying
\begin{equation}
  \label{eq:dipole}
u_n\to 1\text{ a.e., and} \lim_{n\to\infty}\int_\Omega|\nabla u_n|=\Sigma(u).
\end{equation}
\end{proposition}
 For completeness we present the proof of \rprop{prop:dipole}
 in the Appendix.
\vskip 2mm
A basic ingredient in
 the proof of inequality \enquote{$\geq$} in \eqref{eq:19b} is the following proposition which  provides an explicit recipe for
constructing  
\enquote{maximizing sequences} for $\Dist_{\woo}$. In order to describe it 
we first introduce, for each $n\geq 3$, a map 
$T_n\in\text{Lip}(\so;\so)$ with $\deg T_n=1$ by $T_n(e^{\im \theta})=e^{\im
  \tau_n(\theta)}$, with $\tau_n$ defined on $[0,2\pi]$ by setting $\tau_n(0)=0$ and 
\begin{equation}
\label{eq:taun}
  \tau_n'(\theta)= \begin{cases} n,&
  \theta\in(2j\, \pi/n^2, (2j+1)\, \pi/n^2]\\
             -(n-2),&
             \theta\in      ((2j+1)\, \pi/n^2, (2j+2)\, \pi/n^2]
  \end{cases},\ j=0,1,\ldots,n^2-1. 
\end{equation}
\begin{proposition}
  \label{prop:X}
For every $u_0,v_0\in\woo(\Omega;\so)$ such that  $u_0\not\sim v_0$ we have
\be
\label{eq:limnuv}
\lim_{n\to\infty} \frac{d_{\woo}(T_n\circ
  u_0,\mathcal{E}(v_0))}{\Sigma(u_0{\overline v}_0)}=1
\ee
 and the limit is {\em uniform} over all such $u_0$ and
 $v_0$. Consequently
 \begin{equation}
\label{eq:X3}
   \Dist_{\woo}(\mathcal{E}(u_0),\mathcal{E}(v_0))\geq \Sigma(u_0{\overline v}_0).
 \end{equation}
\end{proposition}
\par
As mentioned above, a special case of interest is the distance of a given
$u\in\woo(\Omega;\so)$ to the class $\mathcal{E}$. 
An immediate consequence of \rth{th:main} is that  for every $u\in W^{1,1}(\Omega;{\mathbb S}^1)$ we have
  \begin{equation}
    \label{eq:3}
 \frac{2}{\pi} \Sigma(u)\leq d_{\woo}(u,\mathcal{E})\leq\Sigma(u),
  \end{equation}
 and the bounds are optimal in the sense that
 \begin{align}
 \sup_{u\notin\mathcal{E}}\frac{d_{\woo}(u,\mathcal{E})}{\Sigma(u)}=1, \label{eq:4}
\intertext{and}
\inf_{u\notin\mathcal{E}}\frac{d_{\woo}(u,\mathcal{E})}{\Sigma(u)}=\frac{2}{\pi}.\label{eq:7}
 \end{align}

There are challenging problems concerning the question
whether the supremum and the infimum in the above formulas are
achieved (see \S\ref{subsec:thoughts}). 
\begin{remark}
\label{rem:dist}
Formulas \eqref{eq:3}--\eqref{eq:7} provide  a sharp improvement of the inequality 
\begin{equation}
  \label{eq:sec11-6}
\frac{1}{2}\Sigma(u)\leq d_{\woo}(u,\mathcal{E})\leq \Sigma(u), \ 
\forall\,  u\in \woo(\Omega;\so),
\end{equation}
established in \cite[Sec.~11.6]{bm}.
\end{remark}

 \vskip 2mm
Finally, we turn in Section~\ref{sec:w1p} to the classes in $W^{1,p} (\Omega ; \so)$, $1< p<\infty$,
defined in an analogous way to the $\woo$-case, i.e., using the
equivalence relation 
\begin{equation}
   \label{eq:equivp}
  u\sim v \text{ if and only if }u=e^{\im \varphi}\, v\text{ for some }
 \varphi\in \wop(\Omega;\R).
 \end{equation}
We point out that if $u,v\in
W^{1,p}(\Omega ; \so)$ are equivalent according to the equivalence
relation in \eqref{eq:equiv}, then from the relation $e^{\im\varphi}=u{\overline
  v}$ we deduce that
\begin{equation}
  \label{eq:108}
\nabla\varphi=-\im\overline{u}v\nabla(u{\overline
  v})\in L^p(\Omega;\R^N);
\end{equation}
whence $u\sim v$ according to
\eqref{eq:equivp} as well.  When $p\geq2$ and $\Omega$ is simply connected we
have $W^{1,p} (\Omega ;\so)=\{u\in W^{1,1}(\Omega;{\mathbb S}^1);\,u=e^{\im \varphi}\text{ for some }
 \varphi\in W^{1,p}(\Omega;\R)\}$, see Remark~\ref{rem:5A} below. Therefore, the only cases of interest are:\\[2mm]
(a) general $\Omega$ and $1<p<2$,\\
(b) multiply connected $\Omega$ and $p\geq 2$.\\
In all the theorems below we assume that we are in one of these
situations. The distances
between the classes are defined analogously to
\eqref{eq:10}--\eqref{eq:11} by
 \begin{align}
   \label{eq:87}
   \dist_{\wop}(\mathcal{E}(u_0),\mathcal{E}(v_0)):=\inf_{u\sim u_0}\inf_{v\sim
     v_0} \|\nabla(u-v)\|_{L^p(\Omega)}.
\intertext{and}
\label{eq:1}
   \Dist_{\wop}(\mathcal{E}(u_0),\mathcal{E}(v_0)):=\sup_{u\sim u_0}\inf_{v\sim
     v_0} \|\nabla(u-v)\|_{L^p(\Omega)}.
 \end{align}
 The next result establishes a lower bound for $\dist_{\wop}$:
 \begin{theorem}
   \label{th:lp}
  For every $u_0,v_0\in W^{1,p} (\Omega ; \so)$, $1\le p<\infty$, we have 
  \begin{equation}
    \label{eq:77}
 \dist_{\wop}(\mathcal{E}(u_0),\mathcal{E}(v_0))\geq \left(\frac{2}{\pi}\right)\inf_{w\sim
  u_0{\overline v}_0 }\|\nabla w\|_{L^p(\Omega)}.
  \end{equation}
 \end{theorem}
 \begin{remark}
   \label{rem:sp}
 For $p>1$ the infimum on the R.H.S.~of \eqref{eq:77} is actually a
 minimum; this follows easily from \eqref{eq:108} and the fact that
 $\wop$ is reflexive. 
 \end{remark}
 Note that equality in \eqref{eq:77} holds for $p=1$ by
 \eqref{eq:19a}. An example in \cite[Section~4]{rs} shows that  strict inequality
 \enquote{$>$} may occur in \eqref{eq:77} for a multiply connected domain in
 dimension two and $p=2$. We will show in \S\ref{subsec:strict}  that strict
 inequality may also occur  for simply connected domains when $1<p<2$.
On the positive side, we prove equality in \eqref{eq:77} in the case of the distance to $\mathcal{E}$:
\begin{theorem}
  \label{th:smooth}
 For every $u_0\in W^{1,p}(\Omega;\so)$, $1<p<\infty$, we have
 \begin{equation}
   \label{eq:20}
   \dist_{\wop}(\mathcal{E}(u_0),\mathcal{E})=\left(\frac{2}{\pi}\right)\inf_{w\sim
  u_0}\|\nabla w\|_{L^p(\Omega)}.
 \end{equation}
\end{theorem}
\begin{remark}
  When $p>1$ we do not know general conditions on
  $u_0,v_0\in\wop(\Omega;\so)$ that guarantee equality in
  \eqref{eq:77} (a sufficient condition in the case of multiply
  connected two dimensional domain and $p=2$ is given in \cite[Th.~4]{rs}).
\end{remark}
On the other hand, when $p>1$, $\Dist_{\wop}$ between distinct classes
is infinite:
 \begin{theorem}
   \label{th:no-ub}
  For every $u_0,v_0\in W^{1,p} (\Omega ; \so)$, $1<p<\infty$, such that $u_0\not\sim
  v_0$ we have 
  \begin{equation}
    \label{eq:88}
 \Dist_{\wop}(\mathcal{E}(u_0),\mathcal{E}(v_0))=\infty.
  \end{equation}
 \end{theorem}
\begin{remark}
  \label{rem:5A} 
  There is another natural equivalence relation in $\wop(\Omega;\so)$,
  $1\leq p<\infty$, defined by the homotopy classes, i.e.,
  \begin{equation*}
    u\stackrel{\mathcal{H}}{\sim} v\text{ if and only if } u=h(0)\text{ and
    }v=h(1)\text{ for some }h\in C\left([0,1];\wop(\Omega;\so)\right).
  \end{equation*} 
Homotopy classes have been well-studied (see
\cite{b-li,bm-RACSAM,hang-lin,rub-ster,white}).  Clearly $u\sim
v \Longrightarrow u\stackrel{\mathcal{H}}{\sim} v$ (use the
homotopy $h(t)=e^{\im(1-t)\varphi}v$). 
Note however that when $1\leq
p<2$ the equivalence relation $u\sim v$ is {\em much more restrictive}
than $u\stackrel{\mathcal{H}}{\sim} v$; for example let $\Omega=$unit
disc in $\R^2$, $u(x)=x/|x|$ and $v(x)=(x-a)/|x-a|$ with $0\neq a\in\Omega$,
then $u\not\sim v$ (in fact,
$\dist_{\woo}(\mathcal{E}(u),\mathcal{E}(v))=4|a|>0$ by
\eqref{eq:129} below) while
$u\stackrel{\mathcal{H}}{\sim} v$, e.g., via the homotopy
$h(t)=(x-ta)/|x-ta|, \,0\leq t\leq 1$.
\end{remark}
Part of the results were announced in \cite{bmsCRAS}.
\subsubsection*{Acknowledgments} 
The first author (HB) was partially supported by NSF grant DMS-1207793.  The second author (PM) was partially  supported  by the LABEX MILYON (ANR-10-LABX-0070) of Universit\'e de Lyon,
within the program \enquote{Investissements d'Avenir} (ANR-11-IDEX-0007) operated
by the French National Research Agency (ANR). The third author (IS)
was supported by  the Israel Science Foundation (Grant No. 999/13).

\section{Further comments on $\Sigma(u)$ and $\mathcal{E}(u)$}
\label{sec:comments}
Given $a,b\in\C$, write as usual $a=a_1+\im a_2$, $b=b_1+\im b_2$; we
also identify $a,b$ with the vectors  $a=(a_1,a_2)^T,
b=(b_1,b_2)^T\in\R^2$ and set
\begin{equation}
  \label{eq:126}
  a\land b=a_1b_2-a_2b_1=\text{Im}(\overline{a}b)\in\R.
\end{equation}
\subsection{The distributional Jacobian $Ju$}
For every $u\in\woo(\Omega;\so)$ we consider $u\land\nabla u\in
L^1(\Omega;\R^N)$ defined by its components
\begin{equation}
  \label{eq:138}
  (u\land \nabla u)_j=u\land\frac{\partial u}{\partial x_j}=
u_1\frac{\partial u_2}{\partial x_j}-u_2\frac{\partial u_1}{\partial
  x_j},\  j=1,\ldots,N.
\end{equation}
Since $|u|^2=1$ on $\Omega$ we have
\begin{equation}
  \label{eq:139}
 u_1\frac{\partial u_1}{\partial x_j}+u_2\frac{\partial u_2}{\partial
   x_j}=0 \text{ in }\Omega,
\end{equation}
and thus
\begin{equation}
  \label{eq:140}
 u\land\frac{\partial u}{\partial x_j}=-\im \overline{u}\frac{\partial
   u}{\partial x_j}\text{ in }\Omega;
\end{equation}
in particular,
\begin{equation}
  \label{eq:141}
|u\land\nabla u|=|\nabla u|\text{ in }\Omega.
\end{equation}
The following identities are elementary:
\begin{align}
  (uv)\land\nabla(uv)&=u\land\nabla u+v\land\nabla v, & &  \forall\,
  u,v\in\woo(\Omega;\so),\label{eq:1.4}\\
 e^{\im\varphi}\land \nabla(e^{\im\varphi})&=\nabla\varphi, & &
  \forall\, \varphi\in\woo(\Omega;\R),&\label{eq:1.5}\\
\overline{u}\land \nabla\overline{u}&=-u\land\nabla u,  & &\forall\, 
  u\in\woo(\Omega;\so).\label{eq:1.6}
\end{align}
Finally we introduce, for every $u\in\woo(\Omega;\so)$, its
{\em distributional} Jacobian $Ju$, which is an antisymmetric matrix
with coefficients in $\mathcal{D}'(\Omega;\R)$ defined by
\begin{equation}
  \label{eq:142}
(Ju)_{i,j}:=\frac{1}{2}\left[\frac{\partial}{\partial
    x_i}\left(u\land\frac{\partial u}{\partial x_j}\right)-
\frac{\partial}{\partial x_j}\left(u\land\frac{\partial u}{\partial x_i}\right)\right].
\end{equation}
When $N=2$, $Ju$ is identified with the {\em scalar} distribution
\begin{equation}
  \label{eq:143}
Ju=\frac{1}{2}\left[\frac{\partial}{\partial
    x_1}\left(u\land\frac{\partial u}{\partial x_2}\right)-
\frac{\partial}{\partial x_2}\left(u\land\frac{\partial u}{\partial
  x_1}\right)\right]=\frac{1}{2}\curl\left(u\land\nabla u\right).
\end{equation}
From \eqref{eq:1.4}--\eqref{eq:1.6} we deduce that
\begin{align}
  J(uv)&=Ju+Jv, & &\forall\, u,v\in\woo(\Omega;\so),\label{eq:145}\\
J(\overline{u})&=-Ju, & &\forall\, u\in\woo(\Omega;\so),\label{eq:146}\\
J(e^{\im\varphi})&=0, & &\forall\,
  \varphi\in\woo(\Omega;\R),\label{eq:147}
\end{align}
i.e.,
\begin{equation}
J(u)=0, \ \forall\, u\in\mathcal{E},\label{eq:148}
\end{equation}
and thus
\begin{equation}
u\sim v \Longrightarrow Ju=Jv.\label{eq:149}
\end{equation}
When $\Omega$ is {\em simply connected} the converse is also true, so
that
\begin{equation}
  \label{eq:150}
u\sim v \Longleftrightarrow Ju=Jv;
\end{equation}
in other words,
\begin{equation}
  \label{eq:151}
\mathcal{E}(u)=\{v\in\woo(\Omega;\so);\, Ju=Jv\}.
\end{equation}
This fact is originally due to Demengel~\cite{demengel}, with roots in
\cite{bethuel-IAHP}; simpler proofs can be found in
\cite{bm,bbm-lifting,carbou}. 
\par In order to have a more concrete perception of the equivalence
relation $u\sim v$ it is instructive to understand what it means when
$N=2$ and $\Omega$ is simply connected, for $u,v\in\mathcal{R}$ where  
\begin{equation}
  \label{eq:152}\mathcal{R}=\{u\in\woo(\Omega;\so);\,u\text{ is
    smooth in $\Omega$ except at a finite number of points}\}.
\end{equation}
The class $\mathcal{R}$ plays an important role since it is dense in
$\woo(\Omega;\so)$ (see \cite{bethuelzheng,bm}).
\par If $u\in\mathcal{R}$ then
\begin{equation}
  \label{eq:153}
Ju=\pi\sum_j d_j\delta_{a_j},
\end{equation}
where the $a_j$'s are the singular points of $u$ and $d_j:=\deg(u,a_j)$,
i.e., the topological degree of $u$ restricted to any small circle
centered at $a_j$; see \cite{bcl,bmp,bm} and also \cite[end of
Section~6]{ball} for the special case where $u(x)=x/|x|$. In
particular, when $u,v\in\mathcal{R}$, 
\begin{equation}
  \label{eq:154}
  \begin{aligned}
 u\sim v \Longleftrightarrow [&\text{$u$ and $v$ have the same singularities}\\
 &
\text{and the same degree at each singularity}].
   \end{aligned}
\end{equation}
\subsection{$\Sigma(u)$ computed by duality}
\label{subsec:duality}
An equivalent formula to \eqref{eq:1*} is
\begin{equation}
    \label{eq:122}
    \Sigma(u)=\inf_{\varphi\in\woo(\Omega;\R)}\int_\Omega|u\land
      \nabla u-\nabla\varphi|.
  \end{equation}
Indeed, from \eqref{eq:1.4}--\eqref{eq:1.6} we have
$ue^{-\im\varphi}\land\nabla(ue^{-\im\varphi})=u\land \nabla
u-\nabla\varphi$, and by \eqref{eq:141}, $|\nabla(ue^{-\im\varphi})|=|u\land \nabla
u-\nabla\varphi|$, which yields \eqref{eq:122}.
\par Next we apply the following standard consequence of the
Hahn-Banach theorem:
\begin{equation}
  \label{eq:155}
\dist(p,M)=\inf_{m\in M}\|p-m\|=\max\{<\xi,p>;\,\xi\in M^{\perp},\|\xi\|\leq1\},
\end{equation}
where $E$ is a Banach space, $p\in E$, and $M$ is a linear subspace of
$E$ (see e.g., \cite[Section~1.4, Example 3]{br-FA}).
If we take $E=L^1(\Omega;\R^N)$, $p=u\land\nabla u$,
$M=\{\nabla\varphi;\,\varphi\in\woo(\Omega;\R)\}$, then we have
\begin{equation}
  \label{eq:156}
 M^{\perp}=\{\xi\in L^\infty(\Omega;\R^N);\,\div\xi=0\text{ in
 }\Omega\text{ and }\xi\cdot\nu=0\text{ on }\partial\Omega\},
\end{equation}
where $\nu$ is the outward normal to $\partial\Omega$.
Here the condition [$\div\xi=0\text{ in
 }\Omega\text{ and }\xi\cdot\nu=0\text{ on }\partial\Omega$] is
 understood in the weak sense
 [$\int_\Omega\xi\cdot\nabla\varphi=0,\,\forall\varphi\in\woo(\Omega;\R)$],
 or equivalently,
 [$\int_\Omega\xi\cdot\nabla\varphi=0,\,\forall\varphi\in C^\infty(\overline{\Omega};\R)$].
Inserting \eqref{eq:156} in \eqref{eq:155} yields
\begin{equation}
  \label{eq:157}\Sigma(u)=\max \{\int_\Omega(u\land\nabla
  u)\cdot\xi;\,\xi\in M^{\perp},\|\xi\|_{L^\infty}\leq1\}.
\end{equation}
\par Next we assume that $N=2$ and $\Omega$ is {\em simply
  connected}. We claim that for every $u\in\woo(\Omega;\so)$, 
\begin{equation}
  \label{eq:158}
 \Sigma(u)=\max \{\int_\Omega(u\land\nabla
  u)\cdot\nabla^{\perp}\zeta;\,\zeta\in
  W^{1,\infty}_0(\Omega;\R)\text{ and }\|\nabla\zeta\|_{L^\infty}\leq 1\},
\end{equation}
 where $\nabla^{\perp}\zeta=(-\partial\zeta/\partial
   x_2, \partial\zeta/\partial x_1)$.\\[1mm]
{\em Proof of \eqref{eq:158}.} In view of
\eqref{eq:156}--\eqref{eq:157} it suffices to show that
\begin{equation}
  \label{eq:159}
\{\xi\in L^\infty(\Omega;\R^2);\div\xi=0\text{ in }\Omega\text{ and
}\xi\cdot\nu=0\text{ on }\partial\Omega\}=\{\nabla^{\perp}\zeta;\,\zeta\in W^{1,\infty}_0(\Omega;\R)\}.
\end{equation}
For the inclusion ``$\supset$'', we verify that
\begin{equation*}
\int_\Omega\nabla^{\perp}\zeta\cdot\nabla\varphi=0,~\forall\varphi\in C^\infty(\overline{\Omega};\R);
\end{equation*}
this is clear since $\curl(\nabla\varphi)=0$ and $\zeta=0$ on
$\partial\Omega$.
\par
For the inclusion ``$\subset$", we start with some $\xi\in
L^\infty(\Omega;\R^2)$ such that
\begin{equation}
  \label{eq:161}
\int_\Omega\xi\cdot\nabla\varphi=0,~\forall\varphi\in\woo(\Omega;\R).
\end{equation}
Set ${\bar\xi}:=\begin{cases}\xi,&\text{ in }\Omega\\ 0,&\text{ in
  }\R^2\setminus\Omega\end{cases}$. Then, by \eqref{eq:161},
\begin{equation}
  \label{eq:162}\int_{\R^2}{\bar\xi}\cdot\nabla\Phi=\int_\Omega\xi\cdot\nabla(\Phi|_{\Omega})=0,~\forall\Phi\in C^1_c(\R^2;\R).
\end{equation}
Thus we may invoke the generalized Poincar\'e lemma in $\R^2$ and
conclude that ${\bar\xi}=\nabla^{\perp}{\bar\zeta}$ for some ${\bar
  \zeta}\in W^{1,\infty}(\R^2;\R)$. Clearly,
$\zeta={\bar\zeta}|_\Omega\in W^{1,\infty}(\Omega;\R)$,
$\nabla^{\perp}\zeta=\xi$ and $\zeta$ is constant on $\partial\Omega$
(since $\partial\Omega$ is connected because $\Omega$ is simply
connected).\qed
\begin{remark}
  Equality \eqref{eq:158} is originally due to \cite[Thm~2]{bmp} (with
  a much more complicated proof).
\end{remark}
\par Finally we give a geometric interpretation for $\Sigma(u)$ when
$\Omega\subset\R^2$ is simply connected and $u\in\mathcal{R}$. We
first need some notation. Given $a,b\in\overline{\Omega}$, set
\begin{equation}
\label{eq:163}
  d_\Omega(a,b)=\min\{|a-b|,d(a,\partial\Omega)+d(b,\partial\Omega)\}=\inf_\Gamma\text{length}(\Gamma\cap\Omega),
\end{equation}
where the $\inf_\Gamma$ is taken over all curves $\Gamma\subset\R^2$
joining $a$ to $b$. Clearly $d_\Omega$ is a semi-metric on
$\overline{\Omega}$ ; moreover
\begin{equation*}
  d_\Omega(a,b)=0 \Longleftrightarrow \text{[either $a=b$ or $a,b\in\partial\Omega$]}.
\end{equation*}
Thus we may identify $\partial\Omega$ as a single point in
$\overline{\Omega}$, still denoted $\partial\Omega$.
\par Given $(\vec a,\vec d)=(a_1,a_2,\ldots,a_l,d_1,d_2,\ldots,d_l)$ with
$a_j\in\Omega$ and $d_j\in\Z$, $\forall j$, we set
\begin{equation}
  \label{eq:164}
D=-\sum_{j=1}^l d_j,
\end{equation}
and we consider the collection $(a_1,a_2,\ldots,a_l,\partial\Omega)$
in $\overline{\Omega}$ affected with the integer coefficients
$(d_1,d_2,\ldots,d_l, D)$. We then repeat the points $a_j$'s and
$\partial\Omega$ according to their multiplicities, i.e.,
$d_1,d_2,\ldots,d_l$ and $D$, and we rewrite them
as a collection of $m$ positive points $(P_j)$ and $m$ negative points
$(N_j)$, $1\leq j\leq m$ (this is possible by \eqref{eq:164}). Finally
we define
\begin{equation}
  \label{eq:165}
 L(\vec a,\vec d)=\min_{\sigma\in \mathcal{S}_m}\sum_{j=1}^m d_\Omega(P_j,N_{\sigma(j)}),
\end{equation}
where $\mathcal{S}_m$ denotes the set of permutations of $\{1,2,\dots,m\}$.
\par We are now ready to state our main claim:
  \begin{equation}
    \label{eq:166}
\Sigma(u)=2\pi L(\vec a,\vec d),\ \forall\, u\in\mathcal{R},
  \end{equation}
where the $a_j$'s are the singular points of $u$ and $d_j=\deg(u,a_j)$.
\begin{remark}
  A variant of formula \eqref{eq:166} where $\Omega=\st$ (and
  thus $\partial\Omega=\emptyset$) appears originally in \cite{bmp}, but
  the   core of the proof goes back to \cite{bcl}.
\end{remark}
Here is a sketch of the proof of \eqref{eq:166}. From \eqref{eq:143}
and \eqref{eq:153} we have
\begin{equation}
  \label{eq:167}
-\int_\Omega(u\land\nabla
u)\cdot\nabla^{\perp}\zeta=2\pi\sum_{j=1}^ld_j\zeta(a_j),\ \forall\, \zeta\in W^{1,\infty}_0(\Omega;\R).
\end{equation}
Set
$W^{1,\infty}_{\text{const}}(\Omega;\R)=\{\zeta\in W^{1,\infty}(\Omega;\R);\,\zeta=\text{const}
\text{ on }\partial\Omega\}$ and let $\zeta\in
W^{1,\infty}_{\text{const}}(\Omega;\R)$. From \eqref{eq:167} applied to
$\zeta-\zeta(\partial\Omega)$ we obtain
\begin{equation}
  \label{eq:168}
-\int_\Omega(u\land\nabla
u)\cdot\nabla^{\perp}\zeta=2\pi\left(\sum_{j=1}^l d_j\zeta(a_j)+D\zeta(\partial\Omega)\right).
\end{equation}
Combining \eqref{eq:158} and \eqref{eq:168} we see that
\begin{equation}
  \label{eq:169}
\Sigma(u)=2\pi\max\left\{
\sum_{j=1}^m\left(\zeta(P_j)-\zeta(N_j)\right);\, \zeta\in
W^{1,\infty}_{\text{const}}(\Omega;\R)\text{ and }\|\nabla\zeta\|_{L^\infty}\leq 1\right\}.
\end{equation}
Next we observe that for every $\zeta:\overline{\Omega}\to\R$ the
following conditions are equivalent:
\begin{align}
  \zeta&\in W^{1,\infty}_{\text{const}}(\Omega;\R)\text{ and
         }\|\nabla\zeta\|_{L^\infty}\leq 1\label{eq:170}\\
\intertext{and}
&|\zeta(x)-\zeta(y)|\leq d_\Omega(x,y),\ \forall\, x,y\in\overline{\Omega}.\label{eq:171}
\end{align}
Thus \eqref{eq:169} becomes
\begin{equation}
  \label{eq:172}
\Sigma(u)=2\pi\max\left\{\sum_{j=1}^m\left(\zeta(P_j)-\zeta(N_j)\right);\,\zeta\text{
  satisfying }\eqref{eq:171}\right\}.
\end{equation}
Finally we invoke the formula
\begin{equation}
\label{eq:formula}
  \max\left\{\sum_{j=1}^m\left(\zeta(P_j)-\zeta(N_j)\right);\,\zeta\text{
  satisfying }\eqref{eq:171}\right\}=L({\vec a},{\vec d})
\end{equation}
to conclude that $\Sigma(u)=2\pi L({\vec a},{\vec d})$. 
\par Relation
\eqref{eq:formula} appears originally in \cite[Lemma~4.2]{bcl}. The
proof in \cite{bcl} combines a theorem of Kantorovich with Birkhoff's
theorem on doubly stochastic matrices. An elementary proof of
\eqref{eq:formula}, totally self-contained, is presented in
\cite{br1987}; it is inspired by the proof of the celebrated result of
Rockafellar concerning cyclically monotone operators. Versions of
formula \eqref{eq:formula} have become part of folklore in the optimal
transport community under the name \enquote{Kantorovich duality'}, see
  e.g., \cite{evans,santam,villani}.
\subsection{Optimal lifting}
\label{subsec:optimal-lifting}
It is known (see \cite[Section~6.2]{GMS} and \cite{DI,M,bm})  that  every $u\in\woo(\Omega;\so)$ can be written as
$u=e^{\im\varphi}$ with $\varphi\in BV(\Omega;\R)$. In fact, there are
many such $\varphi$'s in $BV$ and it is natural to introduce the
quantity  
\begin{equation}
  \label{eq:136}
  E(u)=  
  \inf \left\{\int_\Omega |D\varphi|;\, \varphi\in BV(\Omega;\R)
  \text{ such that }u=e^{\im\varphi}\right\}.
\end{equation}
 Then,
 \begin{equation}
   \label{eq:137}
   E(u)=\int_\Omega|\nabla u|+\Sigma(u).
 \end{equation}
Formula \eqref{eq:137} was originally established in \cite{bmp} when
 $N=2$ (and $\Omega=\st$). The nontrivial extension to $N\geq 2$ can
 be deduced from results of Poliakovsky~\cite{pol}, see also \cite{bm}
 for a direct approach.
\subsection{Relaxed energy}
\label{subsec:relaxed}
The {\em relaxed energy} is defined for every $u\in\woo(\Omega;\so)$
by
\begin{equation*}
  R(u)=\inf\left\{\liminf_{n\to\infty}\int_\Omega|\nabla 
  u_n|;\,u_n\in C^\infty(\overline{\Omega};\so),\,u_n\to u\text{
    a.e. on }\Omega\right\},
\end{equation*}
where the first $\inf$ means that the infimum is taken over all
sequences $(u_n)$ in $C^\infty(\overline{\Omega};\so)$ such that
$u_n\to u$ a.e.~on $\Omega$. [In general there is no sequence $(u_n)$
in $C^\infty(\overline{\Omega};\so)$ such that $u_n\to u$ in $\woo$,
unless $Ju=0$. However, it is always possible to find a sequence
$(u_n)$ in $C^\infty(\overline{\Omega};\so)$ such that $u_n\to u$ 
    a.e.~on $\Omega$.] Assume that $\Omega$ is simply connected, then
    \begin{equation*}
      R(u)=\int_\Omega|\nabla u|+\Sigma(u),
    \end{equation*}
see \cite{bmp} for $N=2$ and \cite{bm} for $N\geq 3$.

\section{Motivation}
\label{sec:mot-gen}
In order to illustrate the significance of the results of
\rth{th:main} it is instructive to explain it in a special case
involving  maps with a finite number of singularities. Moreover, this
allows us to compare the problem to an analogous one involving the
Dirichlet energy of
$\st$-valued maps on three dimensional domains, whose study was
initiated in \cite{bcl}. Since in both cases the energy scales like
length, one may expect similar results; as we shall see below the
analogy is not complete. We start with the problem in $\R^3$. Consider for simplicity
$\Omega=B_R(0)\subset\R^3$. Analogously to \eqref{eq:152} we consider
the set $\mathcal{R}$ of maps in $H^1(\Omega;S^2)$ which are smooth on $\overline{\Omega}$,
except at (at most) a finite number of singularities. With each $k$-tuple of
distinct points ${\vec a}=(a_1,\ldots,a_k)\in\Omega^k$ and 
corresponding degrees ${\vec d}=(d_1,\ldots,d_k)\in\Z^k$ we associate the
following class of maps in $\mathcal{R}$:  
\begin{equation}
  \label{eq:5}
\mathcal{E}_{{\vec a},{\vec d}}:=\left\{u\in C^\infty\left(\overline{\Omega}\setminus
  \bigcup_{j=1}^k\{a_j\};{\mathbb S}^2\right);\,\nabla u\in L^2(\Omega)
  \text{ and }\deg(u,a_j)=d_j,\,\forall\, j\right\}.
\end{equation}
[Here, $\deg(u,a_j)=d_j$ means that the restriction of $u$ to any small
sphere around $a_j$ has topological degree $d_j$.] In the case where $k=0$
the resulting class is
$C^\infty(\overline{\Omega};\st)$. There are three natural questions that we want to discuss:
\begin{itemize}
\item [(i)] What is the least energy of a map in $\mathcal{E}_{{\vec
      a},{\vec d}}$
  i.e ., the value of
  \begin{equation}
    \label{eq:6}
   \Sigma_{{\vec a},{\vec d}}^{(2)}:=\inf_{u\in\mathcal{E}_{\vec
       a,\vec d}}\int_\Omega|\nabla u|^2 \text{ ?}
  \end{equation}
\item [(ii)] Consider two sets of distinct points in $\Omega$, 
  ${\vec a}=(a_1,a_2,\ldots,a_k)$ and ${\vec b}=(b_1,b_2,\ldots,b_l)$,
  each with associated vectors of degrees, $\vec d\in\Z^k$ and $\vec
  e\in\Z^l$, respectively. What is
  the $H^1$-distance between 
$\mathcal{E}_{{\vec a},{\vec d}}$ and $\mathcal{E}_{{\vec b},{\vec e}}$
  i.e., analogously to \eqref{eq:10}, 
 \begin{equation}
 \label{eq:9} 
  \dist_{H^1}^2({\mathcal E}_{{\vec a},{\vec d}}, {\mathcal E}_{{\vec
      b},{\vec e}}):=
 \inf_{u\in{\mathcal E}_{{\vec a},{\vec d}}}\, \inf_{v\in{\mathcal E}_{{\vec b},{\vec e}}} \int_\Omega
   |\nabla(u-v)|^2 \text{ ?}
   \end{equation}
 i.e., what is the least energy required to pass from singularities
 located at $\{a_j\}_{j=1}^k$, with degrees  $\{d_j\}_{j=1}^k$, to
 singularities located at $\{b_j\}_{j=1}^l$, with degrees $\{e_j\}_{j=1}^l$?
\item[(iii)] Similarly, by analogy with \eqref{eq:11}, what is the
  value of 
 \begin{equation}
\label{eq:13}
   \Dist_{H^1}^2({\mathcal E}_{{\vec a},{\vec d}}, {\mathcal E}_{{\vec
      b},{\vec e}}):=
   \sup_{u\in{\mathcal E}_{{\vec a},{\vec d}}} \inf_{v\in{\mathcal E}_{{\vec b},{\vec e}}}\int_\Omega |\nabla(u-v)|^2 \text{ ?}
\end{equation}
\end{itemize}
 Question~(i) was originally tackled by
 \cite{bcl}; their motivation came from a question of
 J.~Ericksen concerning the least energy required to produce a liquid crystal
 configuration with prescribed singularities. Quite surprisingly it
 turns out that the value of this least energy can be
 computed explicitly  in terms of {\em geometric} quantities. In the special case
 \eqref{eq:6} their formula becomes 
 \begin{equation}
   \label{eq:16}
\Sigma_{{\vec a},{\vec d}}^{(2)}=8\pi L({\vec a},{\vec d}),
 \end{equation}
where $L(\vec a,\vec d)$
is defined as in \eqref{eq:165}.

 \par
 On the other hand, it seems that Question~(ii) was never treated in
 the literature. Using the results of \cite{bcl} one can show that if
 $(\vec a,\vec d)\neq (\vec b,\vec e)$ then for every {\em fixed} $u\in\mathcal{E}_{\vec a,\vec d}$ we have
 \begin{equation}
   \label{eq:135}
   \dist_{H^1}(u,\mathcal{E}_{\vec b,\vec e})>0.
 \end{equation}
What is quite surprising is that for  {\em all} pairs of classes we have
 \begin{equation}
   \label{eq:18}
 \dist_{H^1}({\mathcal E}_{\vec a,\vec d},{\mathcal E}_{\vec b,\vec e})=0.
 \end{equation}
 The basic ingredient  behind \eqref{eq:18} is the following fact: for
 every pair of integers $d_1\neq d_2$  we have
 \begin{equation}
   \label{eq:28}
   \inf \left\{ \int_{\st} |\nabla(F_1-F_2)|^2;\, F_j\in
   H^1(\st;\st), \deg(F_j)=d_j\;\text{ for }j=1,2\right\}=0.
 \end{equation}
Formula \eqref{eq:28} was established in \cite{ls} (see also
\cite{bms2} for generalizations) following the same idea used by
Brezis and Nirenberg~\cite{BN} in the setting of degree theory in
$H^{1/2}(\so;\so)$. 

As for Question~(iii), the  \enquote{dipole removing} technique of
Bethuel~\cite{bethuel-IAHP} (with roots in \cite{bcl}) can be applied to derive the upper bound 
 \begin{equation}
\label{eq:132}
  \Dist_{H^1}^2({\mathcal E}_{\vec a,\vec d},{\mathcal E}_{\vec b,\vec e}) \leq 8\pi L(\vec c,\vec f),
 \end{equation} 
 where 
\begin{equation}
\label{eq:131}
{\vec c}=(a_1,\ldots,a_k,b_1,\ldots,b_l)\in\Omega^{k+l}\text{ and }
{\vec f}=(d_1,\ldots,d_k,-e_1,\ldots,-e_l)\in\Z^{k+l}.
\end{equation}
 We suspect
 that equality holds in \eqref{eq:132}.
\par It is possible to associate with every $u\in H^1(\Omega;\st)$ a
\enquote{natural} class $\mathcal{E}(u)$, in the spirit of
\eqref{eq:151}. Formulas \eqref{eq:18} and \eqref{eq:132}, as well as
their extensions to arbitrary classes $\mathcal{E}(u)$,
$\mathcal{E}(v)$, are established in \cite{bms17}. We  also
present in \cite{bms17}  evidence that equality holds  in
\eqref{eq:132} by establishing the following analogue of \eqref{eq:4}:
\begin{equation}
  \label{eq:8}
  \sup_{\substack{\vec a,\vec d\\\vec d\neq\boldsymbol 0}}\sup_{u\in \mathcal{E}_{\vec a,\vec d}}
 \frac{d_{H^1}^2(u, C^\infty(\overline\Omega;\st))}{8\pi L(\vec a,\vec
   d)}=1.
\end{equation}
\par
 Next we consider similar questions for $\woo(\Omega;\so)$. For
 simplicity let $\Omega=B_R(0)\subset\R^2$. By analogy with
 \eqref{eq:5}, for
${\vec a}=(a_1,\ldots,a_k)\in\Omega^k$ and ${\vec
  d}=(d_1,\ldots,d_k)\in\Z^k$ we
  consider the
following class of maps in $\mathcal{R}$:  
\begin{equation}
\label{eq:37}
\mathcal{E}_{{\vec a},{\vec d}}:=\left\{u\in C^\infty\left(\overline{\Omega}\setminus
  \bigcup_{j=1}^k\{a_j\};\so\right);\,\nabla u\in L^1(\Omega)
  \text{ and }\deg(u,a_j)=d_j,\,\forall\, j\right\}.
\end{equation}
 The analogous questions to (i)--(iii) are then:
\begin{itemize}
\item [(i')] What is the value of
  \begin{equation}
    \label{eq:57}
   \Sigma_{\vec a,\vec d}^{(1)}:=\inf_{u\in\mathcal{E}_{\vec a,\vec d}} \int_\Omega|\nabla
   u| \text{ ?}
  \end{equation}
\item [(ii')] For any pair $\vec a\in\Omega^k$, $\vec  b\in\Omega^l$ and
  associated vectors of degrees, $\vec d\in\Z^k$ and $\vec e\in\Z^l$,what is
  the $W^{1,1}$-distance between $\mathcal{E}_{\vec a,\vec d}$ and
  $\mathcal{E}_{\vec b,\vec e}$,
\begin{equation}
\label{eq:125}
 \dist_{\woo}({\mathcal E}_{\vec a,\vec d},{\mathcal E}_{\vec b,\vec
   e}):=
\inf_{u\in{\mathcal E}_{\vec a,\vec d}}\, \inf_{v\in{\mathcal E}_{\vec
  b,\vec e}}\int_\Omega
   |\nabla(u-v)|\text{ ?}
  \end{equation}
\item[(iii')] What is the
  value of 
 \begin{equation}
\label{eq:127}
   \Dist_{\woo}({\mathcal E}_{\vec a,\vec d},{\mathcal E}_{\vec b,\vec
     e}):=
   \sup_{u\in{\mathcal E}_{\vec a,\vec d}} \inf_{v\in{\mathcal E}_{\vec b,\vec e}}\int_\Omega |\nabla(u-v)|\text{ ?}
\end{equation}
\end{itemize}
The answer to Question~(i') is given by the results in
\S\ref{subsec:duality}. Indeed,
setting  $\d
u_{\vec a,\vec d}(\zeta):=\Prod_{j=1}^k
\left(\frac{\zeta-a_j}{|\zeta-a_j|}\right)^{d_j}$, 
 we get from  \eqref{eq:166} that
 \begin{equation}
   \label{eq:128}
\Sigma_{\vec a,\vec d}^{(1)}=\Sigma(u_{\vec a,\vec d}) =2\pi L(\vec a,\vec d),
 \end{equation}
which is completely analogous to \eqref{eq:16}. [Here we used the
density of  $\mathcal{E}_{\vec a,\vec d}$  (for the
 $W^{1,1}$-topology) in  $\mathcal{E}(u_{\vec a,\vec d})$ (see \cite{bethuelzheng,bm}).] 
\par On the other hand, the situation with Question (ii') is completely
different. In contrast with \eqref{eq:18}, here
$\dist_{\woo}(\mathcal{E}_{\vec a,\vec d},\mathcal{E}_{\vec b,\vec e})$ is strictly positive when
$(\vec a,\vec d)\neq (\vec b,\vec e)$. The explicit value of this infimum can be computed in terms
of geometric quantities. Actually, \eqref{eq:19a} of \rth{th:main} 
asserts that
  \begin{equation}
\label{eq:129}
  \dist_{\woo}({\mathcal E}_{\vec a,\vec d},{\mathcal E}_{\vec b,\vec
    e})=
\frac{2}{\pi}\,\Sigma(u_{\vec a,\vec d}\,{\overline u}_{\vec b,\vec e})
=4L(\vec c,\vec f),
\end{equation}
where $\vec c$ and $\vec f$ are given by \eqref{eq:131}. Indeed, the
last equality in \eqref{eq:129} follows from \eqref{eq:128} when
applied to the map $u_{\vec a,\vec d}\,{\overline u}_{\vec b,\vec
  e}$ which has singularities precisely at the points $\{c_j\}_{j=1}^{k+l}$, with
associated singularities $\{f_j\}_{j=1}^{k+l}$.
Similarly, the second part of Theorem~1.1, \eqref{eq:19b},  asserts that 
\begin{equation*}
 \sup_{u\in \mathcal{E}_{\vec a,\vec d}} \inf_{v\in \mathcal{E}_{\vec
       b,\vec e}}
 \int_\Omega|\nabla(u-v)|=\Sigma(u_{\vec a,\vec d}\,{\overline u}_{\vec b,\vec e})
=2\pi L(\vec c,\vec f).
\end{equation*}
\par  We also present an interpretation of Theorem~1.1 when $\Omega\subset{\mathbb R}^3$. Fix two disjoint smooth closed oriented curves
$\Gamma_1,\Gamma_2\subset\Omega$ and consider for $j=1, 2$
\begin{equation*}
  \mathcal{E}_{\Gamma_j}=\left\{u\in C^\infty(\overline{\Omega}\setminus\Gamma_j;{\mathbb
    S}^1);\, \nabla u\in L^1(\Omega)\text{ and
  }\deg(u,\Gamma_j)=+1\right\} 
\end{equation*}
($\deg(u,\Gamma_j)=+1$ means that $\deg(u,C_j)=+1$ for every small
circle $C_j\subset\Omega\setminus\Gamma_j$ linking $\Gamma_j$).
In this case Theorem~1.1 asserts that 
\begin{equation*}
\inf_{u\in \mathcal{E}_{\Gamma_1}} \inf_{v\in \mathcal{E}_{\Gamma_2}}
\int_\Omega|\nabla(u-v)|=4\inf_S\text{area}(S\cap\Omega),
\end{equation*}
 where $\inf\limits_S$ is taken over all surfaces $S\subset\R^3$ such
 that $\partial S=\Gamma_1\cup\Gamma_2$, and 
\begin{equation*}
\sup_{u\in \mathcal{E}_{\Gamma_1}} \inf_{v\in \mathcal{E}_{\Gamma_2}}
\int_\Omega|\nabla(u-v)|=2\pi\inf_S\text{area}(S\cap\Omega).
\end{equation*}
For more details on this case, see \cite{abl,bcl,bm}.
\section{Proof of \eqref{eq:19a} in \rth{th:main}}
\label{sec:lb}
\subsection{A basic lower bound inequality}
 We begin with a simple lemma about composition with Lipschitz maps;
 it provides a  very useful device for constructing maps in the same equivalence
     class, or in the class $\mathcal{E}(1)$.
 \begin{lemma}
   \label{lem:T}
 Let
     $T\in \text{Lip}(\so;\so)$ be a map of degree $D$. Then
 \begin{equation}
\label{eq:T}
       T\circ u\sim u^D,\ \forall\,  u\in\woo(\Omega;\so).
 \end{equation}
 \end{lemma}
 \begin{proof}
   Since $T(z){\bar z}^D$  is a Lipschitz self-map of ${\mathbb S}^1$ of zero degree,
   there exists $g\in\text{Lip}({\mathbb S}^1;\R)$ such  that $T(z){\bar z}^D=e^{\im g(z)}$. The
   function $\varphi(x)=g(u(x))$ belongs to $W^{1,1}(\Omega;\R)$ and satisfies
   $T(u(x))=(u(x))^D e^{\im \varphi(x)}$, and \eqref{eq:T} follows by
   the definition of the equivalence relation.
 \end{proof}
 The next simple lemma is essential for the proof of the lower bound
 in \eqref{eq:3}.
 \begin{lemma}
\label{lem:bw}
   For any $w\in W^{1,1}(\Omega;{\mathbb S}^1)$ we have
   \begin{equation}
     \label{eq:12}
\int_\Omega|\nabla(|w-1|)|\geq \frac{2}{\pi}\Sigma(w).
   \end{equation}
 \end{lemma}
 \begin{proof}
   As in \cite{rs}, we define $T:{\mathbb S}^1\to {\mathbb S}^1$ by
   \begin{equation}
\label{eq:41}
     T(e^{\im \varphi}):=e^{\im \theta}\text{ with }\theta=\theta(\va)=\pi\sin(\varphi/2),\;\forall\,\varphi\in(-\pi,\pi],
   \end{equation}
 so that
 \begin{equation}
\label{eq:144}
   |e^{\im \varphi}-1| =2|\sin(\varphi/2)| =\frac{2}{\pi}|\theta|.
 \end{equation}
Clearly $T$ is of class $C^1$ and its degree equals one. We claim that 
\begin{equation}
\label{eq:15*}
|\nabla(|w-1|)|=\frac{2}{\pi} |\nabla(T\circ w)|\text{ a.e.}
\end{equation}
This is a consequence of the standard fact that, if  $F\in\text{Lip}(\so;\R^2)\cap
C^1(\so\setminus\{1\})$ and $w\in\woo(\Omega;\so)$, then $F\circ
w\in\woo(\Omega;\so)$ and, moreover,
\begin{equation*}
  \nabla(F\circ w)=\begin{cases} \dt{F}(w)\nabla w & \text{ a.e.~in }
    [w\neq1]\\
        0 & \text{ a.e.~in }[w=1]
\end{cases}.
\end{equation*}
Integration of \eqref{eq:15*} leads to
\be
\label{eq:15}
\int_\Omega |\na |w-1||=\frac 2\pi \int_\Omega |\na (T\circ w)|.
\ee
By \rlemma{lem:T}, we have $\Sigma(T\circ w)=\Sigma(w)$, and therefore
\eqref{eq:12} 
follows
 from \eqref{eq:15}.
\end{proof}

 \begin{corollary}
   \label{cor:uv}
For every $u,v\in  W^{1,1}(\Omega;{\mathbb S}^1)$ we have
\begin{equation}
  \label{eq:basic}
\int_\Omega |\nabla(u-v)|\geq \frac{2}{\pi}\Sigma(u\overline v).
\end{equation}
 \end{corollary}
\begin{proof}
  Setting $w=u\overline v$ and applying \eqref{eq:12} yields
  \begin{equation*}
    \int_\Omega |\nabla(u-v)|\geq \int_\Omega |\nabla(|u-v|)|=\int_\Omega|\nabla(|w-1|)|\geq\frac{2}{\pi}\Sigma(u\overline v). \qedhere
  \end{equation*}
\end{proof}

\subsection{Proof of \eqref{eq:19a}}
\label{subsec-opt}
We begin by introducing some notation. For an open arc in $\so$ we use
the notation
\begin{equation}
  \label{eq:calA}
\mathcal{A}(\alpha,\beta)=\{\,e^{\im\theta};\,\theta\in(\alpha,\beta)\}
\end{equation}
for any $\alpha<\beta$. We shall also use a specific notation for
half-circles;  for every $\zeta\in \so$ write $\zeta=e^{\im \va}$ with $\va\in (-\pi,\pi]$
and denote $I(\zeta,-\zeta)=\mathcal{A}(\va,\va+\pi)$. 
Note that 
\begin{equation}
z\in I(\zeta,-\zeta)\Longleftrightarrow\zeta\in I(-z,z).\label{eq:I}
\end{equation}
For each $\zeta=e^{\im \va}\in \so$ define a map $P_{\zeta}:\so\to\overline{I(\zeta,-\zeta)}$
by 
\begin{equation}
\label{eq:101}
  P_{\zeta}(z)=\begin{cases}
z, &\text{if } z=e^{\im \theta}\in I(\zeta,-\zeta)\\
e^{\im (2\va-\theta)}=\zeta^{2}\overline{z}, &\text{if } z\notin I(\zeta,-\zeta)
\end{cases},
\end{equation}
 so that for $z\notin I(\zeta,-\zeta)$, $P_{\zeta}(z)$ is the reflection
of $z$ with respect to the line $\ell_{\zeta}=\{t\zeta;\, t\in\mathbb{R}\}$. Next we state
\begin{proposition}
\label{prop:uv}For every $u\in W^{1,1}(\Omega;{\mathbb S}^1)$ we have
\begin{equation}
\label{eq:17}
  \int_{{\so}}\left(\int_{\Omega}|\nabla(u-P_{\zeta}\circ u)|\,dx\right)\,d\zeta=4\int_{\Omega}|\nabla u|\,dx.
\end{equation}
\end{proposition}
 \begin{proof}
For each $\zeta\in {\mathbb S}^1$ set 
 $v_{\zeta}:=P_{\zeta}\circ u.$
By \rlemma{lem:T}, $v_{\zeta}\in\mathcal{E}(1)$, since $\deg P_\zeta=0$. 
We note that
\begin{equation}
\label{hl2}
  z-P_\zeta(z)=\begin{cases}
0, &\text{if } z\in I(\zeta,-\zeta)\\
z-\zeta^2\, \overline z,
&\text{if } z\notin I(\zeta,-\zeta)
\end{cases}.
\end{equation}

Set $w_{\zeta}:=u-v_{\zeta}$. Using \eqref{hl2}, we find that for every $\zeta=e^{\im\va}$ and a.e. $x\in\Omega$ 
we have 
\begin{equation}
\label{hl200}
  \nabla w_{\zeta}(x)=\begin{cases}
0, &\text{if } u(x)\in I(\zeta,-\zeta)\\
\na u(x)-\zeta^2\, \na\overline u(x),
&\text{if } u(x)\notin I(\zeta,-\zeta)
\end{cases}.
\end{equation}


Therefore, for a.e. $x\in\Omega$ we have 
\begin{equation}
\label{hl3}
  |\nabla w_{\zeta}(x)|=\begin{cases}
0,&\text{if } u(x)\in I(\zeta,-\zeta)\\
2|\cos(\theta-\va)||\nabla u(x)|, &\text{if } u(x)=e^{\im \theta}\notin I(\zeta,-\zeta)
\end{cases}.
\end{equation}

Indeed, we   justify \eqref{hl3} e.g. when $\zeta=1$.
In view of \eqref{hl200}, we have to prove that 
\be
\label{hl500}
|\na \textrm{Im}\, u(x)|=|\textrm{Re}\, u(x)|\, |\na u(x)|\ \text{ for a.e. }x.
\ee

If we differentiate the identity $|u|^2\equiv 1$, we obtain\bes
\textrm{Re}\, u\, \na (\textrm{Re}\, u)+\textrm{Im}\, u\, \na (\textrm{Im}\, u)=0\ \text{a.e.};
\ees
this easily implies \eqref{hl500}.
 
 Using \eqref{eq:I} we find that, with $u(x)=e^{\im\theta}$ and 
 \bes
 A(x)=\{ \va\in (-\pi, \pi];\, u(x)\notin
  I(e^{\im \va},-e^{\im \va})\}, 
  \ees
 we have 
 \bes
\begin{aligned}
\int_{\so}\int_{\Omega}|\nabla w_{\zeta}(x)|\, dx\, d\zeta=&\int_{-\pi}^{\pi}\int_{\Omega}\chi_{A(x)}(\va)\, 2|\cos(\theta-\va)||\nabla u(x)|\, dx\, d\va \\
= & \int_{\Omega}|\nabla u(x)|(\int_{\theta}^{\theta+\pi}2|\cos(\theta-\va)|\, d\va)\, dx=4\int_{\Omega}|\nabla u(x)|\, dx,
\end{aligned}
\ees
 which is \eqref{eq:17}. Here
  we have used $\int_{\theta}^{\theta+\pi}2|\cos(\theta-\va)|\, d\va=\int_{0}^{\pi}2|\cos t|\, dt=4.$
\end{proof}
The identity \eqref{eq:17} is a key tool in the proof of \enquote{$\leq$} in
\eqref{eq:19a}. For the convenience of the
reader we shall present first the slightly simpler proof when $v_0=1$.
\begin{proof}[Proof of \enquote{$\leq$} in \eqref{eq:19a}
  for $v_0=1$] By \rcor{cor:uv}
  we have
  \bes
\inf_{u\sim u_0}
d_{W^{1,1}}(u,\mathcal{E}(1))\geq\d\frac{2}{\pi}\Sigma(u_0).
\ees

 Use \eqref{eq:1*} to choose a
  sequence $\{u_n\}\subset{\mathcal E}(u_0)$ with $\lim_{n\to\infty} \int_\Omega|\nabla u_n|
  =\Sigma(u_0)$.
 Use \rprop{prop:uv} to choose $\zeta_n\in {\mathbb S}^1$ such that
 \begin{equation*}
   \int_{\Omega}|\nabla(u_n-P_{\zeta_n}\circ u_n)|\leq\frac{2}{\pi}\int_{\Omega}|\nabla u_n|,
 \end{equation*}
 implying that $\lim_{n\to\infty} d_{\woo}(u_n,\mathcal{E}(1))=\d\frac{2}{\pi}\Sigma(u_0)$.
\end{proof}
Next we turn to the general case.
\begin{proof}[Proof of \enquote{$\leq$} in \eqref{eq:19a}
  for general $v_0$] 
By \eqref{eq:basic} we have
\bes
\dist_{\woo}(\mathcal{E}(u_0),\mathcal{E}(v_0))\geq(2/\pi)\, \Sigma(u_0\overline
v_0),
\ees
 so we need to prove that this is actually an equality.
 By  \rprop{prop:dipole} there exists a sequence $\{w_n\}$
satisfying $w_n\sim u_0{\overline v}_0$ for all $n$, $\lim_{n\to\infty}w_n=1$
a.e., and
\begin{equation}
\int_{\Omega}|\nabla w_n|=\Sigma(u_0\overline{v}_0)+\varepsilon_n,\label{eq:limwn}
\end{equation}
 with $\varepsilon_n\searrow0$. By \rprop{prop:uv} we get
\begin{equation}
\int_{{\mathbb S}^1}\int_{\Omega}|\nabla(w_n-P_\zeta\circ w_n)|\, dx\, d\zeta=4\int_\Omega|\nabla w_n|\, dx=4(\Sigma(u_0\overline{v}_0)+\varepsilon_n).\label{eq:wn-vn}
\end{equation}

Hence, there exists $\zeta_n\in \so_-:=\{z=e^{\im \theta};\,\theta\in [-\pi,0]\}$ such that 
\bes
\int_{\Omega}|\nabla(w_{n}-P_{\zeta_{n}}\circ w_{n})|+\int_{\Omega}|\nabla(w_{n}-P_{-\zeta_{n}}\circ w_{n})|\leq\frac{4}{\pi}(\Sigma(u_{0}\overline{v}_{0})+\varepsilon_{n}).
\ees
By \eqref{eq:basic} we have 
\bes
\frac{2}{\pi}\Sigma(u_{0}\overline{v}_{0})\leq\min\left(\int_{\Omega}|\nabla(w_{n}-P_{\zeta_{n}}\circ
w_{n})|,\int_{\Omega}|\nabla(w_{n}-P_{-\zeta_{n}}\circ w_{n})|\right),
\ees
and thus
\begin{equation}
\lim_{n\to\infty} \int_{\Omega}|\nabla(w_{n}-P_{\zeta_{n}}\circ w_{n})|=\frac{2}{\pi}\Sigma(u_{0}\overline{v}_{0}).\label{eq:zetan}
\end{equation}
 Passing to a subsequence, we may assume $\zeta_{n}\to\zeta\in \so_-$.
Therefore, $P_{\zeta}(1)=1$. Denote $F_{n}:=P_{\zeta_{n}}\circ w_{n}$.
Since $w_{n}\to1$ a.e., we have $\lim_{n\to\infty}F_{n}=\lim_{n\to\infty}P_{\zeta}\circ w_{n}=1$
a.e., and it follows that
\begin{equation}
F_{n}-w_{n}\to0\;\text{a.e.}\label{eq:a.e.}
\end{equation}
 For any $v$ such that $v\sim v_{0}$ we have $vF_{n}\sim v_{0}$,
$vw_{n}\sim u_{0}$ and 
\begin{equation}
\frac{2}{\pi}\Sigma(u_{0}\overline{v}_{0})\leq\int_{\Omega}|\nabla(vF_{n}-vw_{n})|\leq\int_{\Omega}|\nabla(F_{n}-w_{n})|+\int_{\Omega}|\nabla v||F_{n}-w_{n}|.\label{eq:sand}
\end{equation}
 From \eqref{eq:zetan}-\eqref{eq:sand} we deduce that 
\bes
\lim_{n\to\infty}\int_{\Omega}|\nabla(vF_{n}-vw_{n})|=\frac{2}{\pi}\Sigma(u_{0}\overline{v}_{0}),
\ees
and the result follows.
\end{proof}

\section{Proof of \eqref {eq:19b} in
  \rth{th:main}}
\label{sec:19b}
\subsection{An upper bound for $\Dist_{\woo}$}
This short subsection is devoted to the proof of the following
\begin{proposition}
  \label{prop:upper}
 For every $u_0,v_0$ in $W^{1,1}(\Omega,\so)$ we have,
 \begin{equation}
    \label{eq:ub}
\Dist_{\woo} ({\cal E}(u_0),{\cal E}(v_0))=\sup_{u\sim u_0} d_{\woo}(u,\mathcal{E}(v_0))\leq \Sigma(u_0{\overline v}_0).
  \end{equation}
\end{proposition}
\begin{proof}
We adapt an argument from  \cite{bm}.
By \rprop{prop:dipole} 
there exists a sequence $\{w_{n}\}\subset W^{1,1}(\Omega ; \so)$ satisfying
$w_{n}\sim u_{0}\overline{v}_{0}$, $w_{n}\to1$ a.e.,
and $\lim_{n\to\infty}\int_{\Omega}|\nabla w_{n}|=\Sigma(u_{0}\overline{v}_{0})$.
For a given $u\in\mathcal{E}(u_0)$ define $v_{n}=u\overline{w}_{n}$
for all $n$.
 Then, $v_{n}\sim v_{0}$
and 
\bes
\begin{aligned}
d_{\woo}(u,\mathcal{E}(v_0))\leq \int_{\Omega}|\nabla(u-v_{n})|&=\int_{\Omega}|\nabla(u(1-\overline{w}_{n}))|\\
&\leq\int_{\Omega}|1-w_{n}||\nabla u|+\int_{\Omega}|\nabla w_{n}|\to\Sigma(u_{0}\overline{v}_{0}).\hfill
\hskip 18mm\qedhere
\end{aligned}
\ees
\end{proof}
\subsection{A lower bound for $\Dist_{\woo}$}
We begin with the following elementary geometric lemma.
\begin{lemma}
  \label{lem:ineqs1}
 Let $z_1$ and $z_2$ be two points in $\so$ satisfying, for some $\varepsilon\in(0,\pi/2)$,
 \begin{equation}
   \label{eq:55}
 d_{\so}(z_1,z_2)\in(\varepsilon,\pi-\varepsilon).
 \end{equation}
 If the vectors $v_1,v_2\in\R^2$ satisfy 
 \begin{equation}
   \label{eq:54}
v_j\perp z_j,\,j=1,2,
 \end{equation}
then
 \begin{equation}
   \label{eq:45}
 |v_1-v_2|\geq (\sin\varepsilon)|v_j|,j=1,2,
 \end{equation}
and in particular 
\begin{equation}
  \label{eq:52}
|v_1-v_2|^2\geq \left(\frac{\sin^2\varepsilon}{2}\right)(|v_1|^2+|v_2|^2).
\end{equation}
\end{lemma}
 Note that the inequality \eqref{eq:52} can be viewed as a \enquote{reverse
 triangle inequality}.
 \begin{proof}
  From the
   assumptions \eqref{eq:55}--\eqref{eq:54} it follows that
   \begin{equation*}
    <v_1,v_2>\leq (\cos\varepsilon)|v_1||v_2|,
   \end{equation*}
and then 
\begin{equation*}
  |v_1-v_2|^2\geq |v_1|^2+|v_2|^2-2(\cos\varepsilon)|v_1||v_2|\geq
  (\sin\varepsilon)^2|v_j|^2,\ j=1,2.\qedhere
\end{equation*}
 \end{proof}
An immediate consequence of \rlemma{lem:ineqs1} is
\begin{lemma}
  \label{lem:ptwise}
 Let $v,\widetilde u\in\woo(\Omega;\so)$ and denote, for
 $\varepsilon\in(0,\pi/2)$,
 \begin{equation}
   \label{eq:83}
\begin{aligned}
  A_\varepsilon:&=\{x\in\Omega;\, d_{\so}(\widetilde
  u(x),v(x))\in(\varepsilon,\pi-\varepsilon)\}\\
&=\{x\in\Omega;\,2\sin(\varepsilon/2)<|\widetilde
  u(x)-v(x)|< 2\cos(\varepsilon/2)\}.
 \end{aligned} 
 \end{equation}
 Then
 \begin{equation}
   \label{eq:2}
   |\nabla(\widetilde u-v)|\geq (\sin\varepsilon)|\nabla\widetilde
   u|\ \text{a.e.~in }A_\varepsilon.
 \end{equation}
\end{lemma}
\begin{proof}
Since $v\perp v_{x_i}$ and $\widetilde u\perp {\widetilde u}_{x_i}$ a.e.~on $\Omega$ for $i=1,\ldots,N$, we may
apply \rlemma{lem:ineqs1} with $z_1={\widetilde
  u}(x),z_2=v(x),v_1={\widetilde u}_{x_i}(x)$
and $v_2=v_{x_i}(x)$ to obtain  
\begin{equation*}
  |{\widetilde u}_{x_i}-v_{x_i}|^2\geq
(\sin\varepsilon)^2|{\widetilde
   u}_{x_i}|^2,\;\text{ a.e. in }A_\varepsilon,~i=1,\ldots,N.
\end{equation*}
 Summing over $i$ yields \eqref{eq:2}.
\end{proof}
The next lemma is the main ingredient in the proof of \rprop{prop:X}.
\begin{lemma}
  \label{lem:main}
 Let $u,{\widetilde u},v\in \woo(\Omega;\so)$, $\varepsilon\in(0,\pi/20)$
 and $A_\varepsilon$ as in \eqref{eq:83}. Assume that
\begin{equation}
\label{eq:14}
|u(x)-\widetilde u(x)|\leq \ve,\  \fo x\in\Omega.
\end{equation}
Then,
\begin{equation}
\label{eq:34}
\int_{A_\varepsilon} |\na (v-\widetilde u)|\ge (1-6\ve)\Sigma (v\overline{u})- 2\int_{A_\varepsilon} |\na u|.
\end{equation}
\end{lemma}
\begin{proof}
  Note first that \eqref{eq:14} implies that $\wtu\sim u$. Indeed, the
  image of the map $\wtu\, \overline{u}$ is contained in an arc of $\so$
  of length$\leq 2\arcsin(\varepsilon/2)$, so there exists
  $\varphi\in\woo(\Omega;\R)$ such that $\wtu=e^{\im\varphi}u$. Hence,
  setting 
  $w:=v/u=v\, {\overline u}$ and $\wtw:=v/\wtu$, we have also $\wtw\sim
  w$. Consider the map
 \begin{equation}
\label{eq:Wn}
   W:=\overline{u}(v-\wtu)+1=w+(1-\wtu/u).
 \end{equation}
By the triangle inequality,
\begin{equation*}
 |\nabla W|= |\nabla\left(\overline{u}(v-\wtu)\right)|\leq 2|\nabla\overline{u}|+|\nabla(v-\wtu)|,
\end{equation*}
 whence
\begin{equation}
  \label{eq:59}
 \int_{A_\varepsilon}|\nabla(v-\wtu)|\geq \int_{A_\varepsilon}|\nabla
 W|-2\int_{A_\varepsilon}|\nabla u|.
\end{equation}
 By \eqref{eq:14}, $|W-w|=|1-\wtu/u|=|u-\wtu|\leq\varepsilon$ in
 $\Omega$. Hence
\begin{equation}
  \label{eq:60}
 \left||W|-1\right|\leq |W-w|\leq \varepsilon~\text{ in }\Omega,
\end{equation}
 and also
 \begin{equation}
   \label{eq:63}
|\wtw-w|=|\wtu-u|\leq \varepsilon~\text{ in }\Omega.
 \end{equation}
Consider the map $\widetilde W:=W/|W|$,  which thanks to
\eqref{eq:60}  belongs to
$W^{1,1}(\Omega;\so)$. Furthermore, again by \eqref{eq:60},
\begin{equation}
\label{eq:65}
  |\widetilde W-w|\leq |\widetilde W-W|+|W-w|\leq 2\varepsilon\ \text{in }\Omega,
\end{equation}
implying in particular that 
 \begin{equation}
\label{eq:62}
\widetilde W\in\mathcal{E}(w).
\end{equation}
Combining \eqref{eq:65} with \eqref{eq:63} yields
 \begin{equation}
\label{eq:64}
   |\widetilde W-\wtw|\leq 3\varepsilon\ \text{and}\  d_{\so}(\widetilde W,\wtw)\leq 6\varepsilon\ \text{in }\Omega.
 \end{equation}
 A direct consequence of \eqref{eq:60} is the
 pointwise inequality in $\Omega$ 
\begin{equation*}
  |\nabla W|\geq (1-\varepsilon) |\nabla \widetilde W|,
\end{equation*}
 which together with \eqref{eq:59} yields
\begin{equation}
  \label{eq:61}
 \int_{A_\varepsilon}|\nabla(v-\wtu)|\geq  (1-\varepsilon)\int_{A_\varepsilon}|\nabla
 \widetilde W|-2\int_{A_\varepsilon}|\nabla u|.
\end{equation}
Since
\begin{equation*}
  A_\varepsilon=\{x\in\Omega;\,\wtw(x)\in\mathcal{A}(\varepsilon,\pi-\varepsilon)\cup\mathcal{A}(\pi+\varepsilon,2\pi-\varepsilon)\}\ \text{(see
    \eqref{eq:83} and \eqref{eq:calA})},
\end{equation*}
we deduce from \eqref{eq:64} that
\begin{equation}
  \label{eq:67}
 B_{\varepsilon}:=\{x\in\Omega;\,\widetilde
 W(x)\in\mathcal{A}(7\varepsilon,\pi-7\varepsilon)\cup\mathcal{A}(\pi+7\varepsilon,2\pi-7\varepsilon)\}\subseteq A_\varepsilon.
\end{equation}
For each $\delta\in(0,\pi/2)$ consider the map $K_\delta:\so\to\so$
defined by 
\begin{equation}
\label{eq:Keps}
K_\delta(e^{\im \theta}):=\begin{cases}
1,&\text{if }-\delta\leq\theta < \delta\\
e^{\im\pi(\theta-\delta)/(\pi-2\delta)},&\text{if } \delta\leq
  \theta<\pi-\delta\\
-1, &\text{if } \pi-\delta\leq\theta< \pi+\delta\\
-e^{\im\pi(\theta-\pi-\delta)/(\pi-2\delta)},& \text{if }\pi+\delta\leq
  \theta< 2\pi-\delta\\
\end{cases}.
 \end{equation} 
Clearly $K_\delta\in\text{Lip}(\so;\so)$ with
$\|\dt{K}_\delta\|_\infty=\pi/(\pi-2\delta)$ and
$\deg(K_\delta)=1$. Therefore, by \eqref{eq:62} and \rlemma{lem:T} 
\begin{equation}
\label{eq:69}
   w_1:=K_{7\varepsilon}\circ\widetilde
W\in\mathcal{E}(w).
\end{equation}
Note that by definition, $\nabla w_1=0$ a.e. on
$\Omega\setminus B_\varepsilon$, so by \eqref{eq:67} and \eqref{eq:69} we have
 \begin{equation}
   \label{eq:68}
\int_{A_\varepsilon} |\nabla\widetilde W|\geq\int_{B_\varepsilon} |\nabla\widetilde W|\geq (1-5\varepsilon)
\int_{B_\varepsilon}|\nabla w_1|=(1-5\varepsilon) \int_\Omega
|\nabla w_1|\geq (1-5\varepsilon)\Sigma(w).
 \end{equation}
Plugging \eqref{eq:68} in \eqref{eq:61} yields
\begin{equation}
  \label{eq:70}
\begin{aligned}
 \int_{A_\varepsilon}|\nabla(v-\wtu)|&\geq
 (1-\varepsilon)\int_{B_\varepsilon}|\nabla
 \widetilde W|-2\int_{A_\varepsilon}|\nabla u|\\
&\geq (1-\varepsilon)(1-5\varepsilon)\Sigma(w)-2\int_{A_\varepsilon}|\nabla u|
\geq (1-6\varepsilon)\Sigma(w)-2\int_{A_\varepsilon}|\nabla u|,
\end{aligned}
\end{equation}
 and \eqref{eq:34} follows. 
\end{proof}
The next result is a direct consequence of  \rlemma{lem:main}.
\begin{corollary}
\label{cor:X}
 There exists a universal constant $C$ such that for every $\varepsilon>0$ we have
\begin{equation}
\label{6} n\ge 1/\ve^2\implies \int_\Omega|\na (T_n\circ u)-v|\ge
(1-C\ve)\, \Sigma (u{\overline v)},\ \fo u, v\in\woo(\Omega;\so).
\end{equation}
\end{corollary}
\begin{proof}
We shall use two basic
properties of $T_n$: 
\begin{align}
 \label{eq:useful1}
& d_{\so}(x,T_n(x))\leq \frac{\pi(n-1)}{n^2},\ \fo x\in\so,\\
&|\dt{T}_n|\geq n-2\text{ a.e.~in }\so.\label{eq:useful2}
  \end{align}
 Clearly it suffices to consider $\varepsilon<\pi/20$. Hence for
 $n\ge 1/\ve^2$ we can apply
\rlemma{lem:main} with $\widetilde u:=T_n\circ u$ (thanks to
\eqref{eq:useful1}).
By \eqref{eq:useful2} we have 
\be
\label{7}
|\na (T_n\circ u)|\geq (n-2)|\na u|\ \text{a.e. on }\Omega,
\ee
so combining \eqref{eq:2} and \eqref{eq:34} gives (recall that
$A_\varepsilon$ is defined in \eqref{eq:83}):
\bes
\begin{aligned}
\int_{A_\varepsilon}|\na (T_n\circ u-v)|&\ge (1-6\ve)\Sigma
(u\overline{v})- \frac{2}{(n-2)\,
  \sin\varepsilon}\int_{A_\varepsilon}|\na (T_n\circ u-v)|\\
&\ge  (1-6\ve)\Sigma(u\overline{v})-
\frac{3}{n\varepsilon}\int_{A_\varepsilon}|\na (T_n\circ u-v)|;
\end{aligned}
\ees
this leads easily to \eqref{6}.
\end{proof}
\begin{proof}[Proof of \rprop{prop:X}]
Recall (see \eqref{eq:ub}) that 
 \begin{equation}
    \label{eq:ub-again}
\Dist_{\woo}(\mathcal{E}(u_0),\mathcal{E}(v_0))=\sup_{u\sim u_0} d_{\woo}(u,\mathcal{E}(v_0))\leq \Sigma(u_0{\overline v}_0),
  \end{equation}
 and in particular, $\forall\,  n\geq 3$,
 \begin{equation}
   \label{eq:47}
    d_{\woo}(T_n\circ u_0,\mathcal{E}(v_0))\leq \Sigma(u_0{\overline v}_0).
 \end{equation}
 On the other hand, from \rcor{cor:X} we know that,
 $\forall\, \varepsilon>0$, $\forall\,  n\geq 1/\varepsilon^2$,
 \begin{equation}
   \label{eq:73}
d_{\woo}(T_n\circ u_0,\mathcal{E}(v_0))\geq (1-C\varepsilon)\Sigma(u_0{\overline v}_0).
 \end{equation}
We conclude combining \eqref{eq:47} and \eqref{eq:73}.
\end{proof}
\begin{proof}[Proof of \eqref{eq:19b}]
 Use \eqref{eq:X3} and \eqref{eq:ub-again}.
\end{proof}
\subsection{About equality cases in \eqref{eq:3}}
\label{subsec:thoughts}
 It is interesting to decide whether  there exist  maps
 $u\in\woo(\Omega;\so)$ for which equality holds in any of the two 
 inequalities in \eqref{eq:3}. Consider the following properties of a
 smooth bounded domain $\Omega$ in $\R^N$, $N\geq2:$\\[2mm]
\noindent $(\text{P}_1)$~There exists $u\in\woo(\Omega;\so)$ such that
\begin{equation}
  \label{eq:P1}
 \int_\Omega|\nabla u|=\Sigma(u)>0.
\end{equation}
\noindent $(\text{P}_2)$~There exists  $u\in\woo(\Omega;\so)$ such
that 
\begin{equation}
  \label{eq:P2}
d_{\woo}(u,\mathcal{E}(1))=\Sigma(u)>0.
\end{equation}
\noindent $(\text{P}_2^*)$~There exist  $u\in\woo(\Omega;\so)$ with
$\Sigma(u)>0$ and 
$v\in\mathcal{E}(1)$ for which
\begin{equation}
   \label{eq:P2*}
\int_\Omega|\nabla(u-v)|=d_{\woo}(u,\mathcal{E}(1))=\Sigma(u).
\end{equation}
\noindent $(\text{P}_3)$~There exists $u\in\woo(\Omega;\so)$ such
that
\begin{equation}
  \label{eq:P3}
d_{\woo}(u,\mathcal{E}(1))=\frac{2}{\pi}\Sigma(u)>0.
\end{equation}
\noindent $(\text{P}_3^*)$~There exist  $u\in\woo(\Omega;\so)$ with
$\Sigma(u)>0$ and 
$v\in\mathcal{E}(1)$ for which
\begin{equation}
   \label{eq:P3*}
\int_\Omega|\nabla(u-v)|=d_{\woo}(u,\mathcal{E}(1))=\frac{2}{\pi}\Sigma(u).
\end{equation}
\vskip 2mm
Very little is known about domains satisfying any of the
above properties. The unit disc $\Omega=B(0,1)$ in $\R^2$ is an example of a domain
for which $(\text{P}_1)$ is
satisfied. Indeed, for $u=x/|x|$ it is straightforward  that 
 \begin{equation*}
   \Sigma\left(\frac{x}{|x|}\right)=2\pi=\int_{\Omega} \left|\nabla\left(\frac{x}{|x|}\right)\right|\ (\text{see also \cite{bmp,bm}})
 \end{equation*}
 whence $(\text{P}_1)$ holds. In view of the following proposition we
 know that 
$(\text{P}_3^*)$ is also satisfied for $\Omega=B(0,1)$ in $\R^2$. 
\begin{proposition}
 \label{prop:P}
Properties $(\text{P}_1)$  and 
$(\text{P}_3^*)$ are equivalent. More precisely, 
 let $u\in W^{1,1}(\Omega;{\mathbb S}^1)$ with $\Sigma(u)>0$. Then, the following are
 equivalent: 
 \begin{itemize}
 \item [(a)] $u$ satisfies \eqref{eq:P1}.
 \item [(b)] There exist  $u_0\in\mathcal{E}(u)$  and $v\in\mathcal{E}(1)$ such
    that $\d\int_\Omega |\nabla(u_0-v)| =\frac{2}{\pi}\Sigma(u)$.
 \end{itemize}
 \end{proposition} 
\begin{proof}[Proof of \enquote{$(a)\Longrightarrow (b)$}]
 Use \rprop{prop:uv} to find $\zeta_0\in\so$ such that
  $v=P_{\zeta_0}\circ u\in\mathcal{E}(1)$ satisfies 
  \begin{equation}
\label{eq:82}
    d_{\woo}(u,\mathcal{E}(1))\leq \int_\Omega |\nabla(u-v)|\leq \frac{2}{\pi}\int_\Omega |\nabla u|= \frac{2}{\pi}\Sigma(u),
  \end{equation}
 and the result follows, with $u_0=u$, since  by \eqref{eq:7} we have
\begin{equation}
  \label{eq:35}
d_{\woo}(u,\mathcal{E}(1))\geq\frac{2}{\pi}\Sigma(u).
\end{equation}

  \medskip
  \noindent
{\it Proof of \enquote{$(b) \Longrightarrow (a)$}.} Let $u_0$ and $v$ be as in
statement {\it (b)}. Set $w_0:=u_0{\overline v}$, so that $w_0\sim u_0$. By
assumption and \eqref{eq:12} we have:
\begin{equation}
\label{eq:40}
 \frac{2}{\pi}\Sigma(u_0)\leq 
 \int_\Omega |\nabla(|w_0-1|)|=\int_\Omega |\nabla(|u_0-v|)|
  \leq 
  \int_\Omega |\nabla(u_0-v)|=\frac{2}{\pi}\Sigma(u_0).
\end{equation}
 Set $w_1:=T\circ w_0$, where $T:{\mathbb S}^1\to {\mathbb S}^1$ is given by
 \eqref{eq:41}. By \rlemma{lem:T}, $w_1\sim w_0\sim u_0$,  and by
 \eqref{eq:15} and \eqref{eq:40} we obtain that
 \begin{equation*}
   \int_\Omega |\nabla w_1|=\frac{\pi}{2}\int_\Omega |\nabla(|w_0-1|)|=\Sigma(u).\qedhere
 \end{equation*}
 \end{proof} 
We do not know any domain $\Omega$ in $\R^2$ for which \eqref{eq:P1}
fails and we ask: 
\begin{open-problem}
\label{OP1}
Is there a domain  $\Omega$ in $\R^N$, $N\geq2$, for
which property $(\text{P}_1)$ (respectively, $(\text{P}_3)$) does not hold?
\end{open-problem} 
\noindent It seems plausible that if $\Omega$ is the interior of a non
circular ellipse, then $(\text{P}_1)$ and $(\text{P}_3)$ fail. We
also do not know whether properties $(\text{P}_3)$ and $(\text{P}_3^*)$ are equivalent.
\\[2mm]

 Concerning properties $(\text{P}_2)$ and $(\text{P}_2^*)$ we know
 even less:
\begin{open-problem}
\label{OP2}
Is there a domain $\Omega$ for which $(\text{P}_2)$ holds
(respectively, fails)?
\end{open-problem}
\noindent We suspect that  $(\text{P}_2)$ 
 is satisfied in {\em every
} domain, but we do not know  {\em any} such domain. In particular, we
do not know what happens when $\Omega$ is a disc in $\R^2$.
\section{Distances in $W^{1,p}(\Omega;\so)$, $1<p<\infty$}
\label{sec:w1p}
Throughout this section we study classes in $W^{1,p}(\Omega;\so)$, where
$1<p<\infty$ and
$\Omega$ is a smooth bounded domain in $\R^N$, $N\geq
2$. We give below the proofs of the results stated in the
Introduction.
\subsection{Proof of \rth{th:lp}}
\begin{proof}[Proof of \rth{th:lp}]
  The result is a direct consequence of the following analog of
  \rcor{cor:uv}: for every $u,v\in  W^{1,p}(\Omega;{\mathbb S}^1)$ we have
\begin{equation}
  \label{eq:basicp}
 \|\nabla(u-v)\|_{L^p(\Omega)}\geq \left(\frac{2}{\pi}\right)\inf_{w\sim
   u\overline v}\|\nabla w\|_{L^p(\Omega)}.
\end{equation}
The proof of \eqref{eq:basicp} uses an  argument identical to the one used in the proofs of
   \rlemma{lem:bw} and \rcor{cor:uv}. Indeed, we first note that
\begin{equation}
\label{eq:90}
    \int_\Omega |\nabla(u-v)|^p\geq \int_\Omega
    |\nabla(|u-v|)|^p=\int_\Omega|\nabla(|u\overline v-1|)|^p.
  \end{equation}
 Next, by \eqref{eq:15*} we have
 \begin{equation}
   \label{eq:89}
\int_\Omega|\nabla(|u\overline v-1|)|^p=\left(\frac{2}{\pi}\right)^p
\int_\Omega|\nabla(T\circ (u\overline v))|^p\geq \left(\frac{2}{\pi}\right)^p\inf_{w\sim u\overline v}\int_\Omega|\nabla w|^p.
 \end{equation}
The result clearly follows by combining \eqref{eq:90} with \eqref{eq:89}.
 \end{proof}
\subsection{Proof of \rth{th:smooth}}
We shall need the
following technical lemma.
\begin{lemma}
  \label{lem:rs}
 For every $w_0\in\wop(\Omega;\so)$ we have
 \begin{equation}
   \label{eq:102}
\inf_{w\sim w_0} \| \nabla(|w-1|)\|_{L^p(\Omega)}=\left(\frac{2}{\pi}\right)\inf_{w\sim w_0} \| \nabla w\|_{L^p(\Omega)}.
 \end{equation}
\end{lemma}
\begin{proof}
  The inequality \enquote{$\geq$} follows from \eqref{eq:89} (taking $v=1$)
  so it remains to prove the reverse inequality. The argument is
  almost identical to the one used in the proof of \cite[Prop 3.1]{rs}; we
  reproduce the argument for the convenience of the reader. 
 We shall need the  inverse $S:=T^{-1}$ of $T:\so\to\so$ that was defined in \eqref{eq:41}.
It is given by: 
\bes
S(e^{\im\theta})=e^{\im\phi}, ~\text{ with }
\phi=2\sin^{-1} (\theta/\pi),\,\forall\,\theta\in(-\pi,\pi].
\ees
This map belongs to $C(\so;\so)\cap C^1(\so\setminus\{-1\};\so)$ but
it is not Lipschitz.
 We therefore define, for each small $\varepsilon>0$, an approximation $S_\varepsilon$
 by:
\begin{equation}
\label{eq:Se}
S_\varepsilon(e^{\im\theta})=e^{\im\phi}~\text{ with }\phi=2\sin^{-1}\left(J_\varepsilon({\theta}/{\pi})\right),~\forall\,\theta\in(-\pi,\pi],
\end{equation}
 where $J_\varepsilon$ satisfies:
\begin{equation}
   \label{eq:Ieps}
 \begin{aligned}
&J_\varepsilon(\pm1)=\pm1,\ J'_\varepsilon(\pm1)=0,\\
&J_\varepsilon(t)=t,\text{ for }|t|\leq 1-\varepsilon,\\
&0<J'_\varepsilon(t)<c_0,\text{ for }|t|<1,\\
&\frac{c_1}{\varepsilon}\leq |J''_\varepsilon(t)|\leq \frac{c_2}{\varepsilon},\text{ for } 1-\frac{\varepsilon}{2}\leq |t|\leq 1,\\
\end{aligned}
 \end{equation}
for some positive constants $c_0,c_1,c_2$ (independent of
$\varepsilon$). Clearly $S_\varepsilon\in C^1(\so;\so)$ with $\deg(S_\varepsilon)=1$, so by \rlemma{lem:T}, for any
$w\in\mathcal{E}(w_0)$ we have $w_\varepsilon:=S_\varepsilon\circ w\in
\mathcal{E}(w_0)$.
 Since
 $|S_\varepsilon(e^{\im\theta})-1|=2|J_\varepsilon(\theta/\pi)|$ it follows from \eqref{eq:Ieps} that
 \begin{equation}
\label{eq:103}
   \left|\frac{d}{d\theta}\left(|S_\varepsilon(e^{\im\theta})-1|\right)\right|\leq C,~\forall\,\theta,\,\forall\,\varepsilon.
 \end{equation}
Put
$A_\varepsilon:=\{x\in\Omega:\,w(x)\in\mathcal{A}(-\pi(1-\varepsilon),\pi(1-\varepsilon))\}$. By
\eqref{eq:15*},
\begin{equation*}
   \left|\nabla|w_\varepsilon-1|\right|=\frac{2}{\pi}|\nabla
   w|\ \text{a.e. on }A_\varepsilon,
\end{equation*}
 while, by \eqref{eq:103},
 \begin{equation*}
    \left|\nabla|w_\varepsilon-1|\right|\leq C|\nabla w|\ \text{a.e. on }\Omega.
 \end{equation*}
Therefore, by dominated convergence,
\begin{equation*}
  \lim_{\varepsilon\to 0}\int_\Omega
  \left|\nabla|w_\varepsilon-1|\right|^p=\left(\frac{2}{\pi}\right)^p
  \int_\Omega |\nabla w|^p,
\end{equation*}
 and since the above is valid for any $w\in \mathcal{E}(w_0)$, the
 inequality \enquote{$\leq$} in \eqref{eq:102} follows.
\end{proof}
The next lemma is the main ingredient of the proof of \rth{th:smooth}.
\begin{lemma}
  \label{lem:main-lemma}
 For every $w\in W^{1,p}(\Omega;\so)$ and 
$0<\delta<1$ there exist a set $A=A(w,\delta)\subset\Omega$  and
 two functions $w_0, w_1\in W^{1, p}(\Omega ; \so)$
such that:\\
(i)  $w_1={\overline w}_0$ in $\Omega\setminus A$;\\
(ii) $w_0=w_1$ in $A$;\\
(iii)  $w_1\in\mathcal{E}(w)$ and $w_0\in\mathcal{E}(1)$;\\
(iv) $\int_\Omega |\nabla(w_1-w_0)|^p\leq (1+C_p\delta)\int_\Omega |\na |w-1||^p$;\\
(v) $\int_\Omega|\nabla w_1|^p=\int_\Omega|\nabla w_0|^p\leq C(\delta, p)\int_\Omega |\na w|^p$.
\end{lemma}
\begin{proof}
Let $I$ denote the open arc of $\so$, $I:=\mathcal{A}(2\pi-\delta,2\pi)=\{e^{\im \theta}:\,\theta\in(2\pi-\delta,2\pi)\}$,
and let  $A:=w^{-1}(I)$.
Define $T_1:\so\to\so$ by  
\begin{equation*}
T_1(e^{\im \theta})=\begin{cases}
e^{\im \pi\theta/(2\pi-\delta)},&\text{if } 0\leq\theta\leq2\pi-\delta\\
 (-1)e^{\im\pi(\theta-(2\pi-\delta))/\delta},&\text{if }2\pi-\delta<\theta<2\pi
\end{cases}.
 \end{equation*} 
Note that the image of $T_1$, restricted to the arc $\so\setminus I$
is $\so_{+}$, and that  $\so_{+}$ is covered  counterclockwise. Similarly, on
the arc $I$, the image of $T_1$ is $\so_{-}$, covered again counterclockwise. It follows that
$\deg(T_1)=1$.
 Next we define $T_0:\so\to\so$ by $P_{-1}\circ T_1$ (see \eqref{eq:101}), or
 explicitly by
\begin{equation*}
T_0(e^{\im \theta}):=\begin{cases}
\overline{T_1(e^{\im \theta})}, &\text{if } 0\leq\theta\leq2\pi-\delta\\
T_1(e^{\im \theta}),&\text{if }2\pi-\delta<\theta<2\pi
\end{cases}.
 \end{equation*} 
  Clearly $\deg(T_0)=0$.
 Define $w_0:=T_0\circ w$ and $w_1:=T_1\circ w$.\\

Properties {\it (i)--(ii)} are  direct consequences of the definition of
$w_0,w_1$. The fact that $w_0\in\mathcal{E}(1)$ and
$w_1\in\mathcal{E}(w)$ (i.e., property {\it (iii)}) follows from \rlemma{lem:T}.
Since $T_0$ and $T_1$ are Lipschitz maps (actually, piecewise smooth,
with a single corner at $z=e^{\im(2\pi-\delta)}$), 
the chain rule implies that 
\bes
|\na w_0|=|\na w_1|\leq \begin{cases} (\pi/\delta)\, |\na w|,&\text{a.e. in }A\\
(\pi/(2\pi-1))\, |\na w|,&\text{a.e. in }\Omega\setminus A
\end{cases},
\ees
 whence property {\it (v)}.
Finally, in order to verify property {\it (iv)} we first notice that on
$\Omega\setminus A$ we have 
\bes
\widetilde w:=w_1\, \overline{w}_0=w_1^2=Q\circ w,
\ees
 where $Q(e^{\im\theta}):=e^{2\im\theta\pi/(2\pi-\delta)}$ for $\theta\in(0,2\pi-\delta)$.
Therefore,
\bes
\begin{aligned}
\int_\Omega |\na(w_1-w_0)|^p&=\int_{\Omega\setminus
  A}|\na(w_1-w_0)|^p=\int_{\Omega\setminus A}\left|\na
|w_1-w_0|\right|^p=\int_{\Omega\setminus A}\big|\na |\widetilde w-1|\big|^p\\
&\leq (1+C_p\delta)\int_{\Omega\setminus A}\big|\na | w-1|\big|^p\leq (1+C_p\delta)\int_{\Omega}\left|\na | w-1|\right|^p. 
\end{aligned}\qedhere
\ees
\end{proof}
We are now in a position to present the
\begin{proof}[Proof of \rth{th:smooth}]
  In view of \rth{th:lp} we only need to prove the inequality
  \enquote{$\leq$}   in \eqref{eq:20}. For any $w\in{\cal E}(u_0)$ we may apply
  \rlemma{lem:main-lemma} with a sequence $\delta_n\to0$ to obtain that
  \begin{equation*}
    d_{W^{1,p}}({\cal E}(u_0),{\cal E}(1))\leq \inf_{w\sim u_0} \|\nabla|w-1|\|_{L^p(\Omega)},
  \end{equation*}
 and the result follows from \rlemma{lem:rs}.
\end{proof}
\subsection{Proof of \rth{th:no-ub}} 
We next turn to the unboundedness of the $\Dist_{W^{1,p}}$-distance
between distinct classes. 
\begin{proof}[Proof of \rth{th:no-ub} when $u_0=1$]
 For every $n\geq 1$ let
\bes
u_n:=e^{\im n x_1}\ \text{(we write ${x}=(x_1,\ldots,x_N)$),}
\ees
 so clearly
$u_n\in C^\infty(\Omega;\so)\subset \mathcal{E}(1)$. We claim that
\begin{equation}
  \label{eq:50}
 \lim_{n\to\infty} d_{\wop}(u_n,\mathcal{E}(v_0))=\infty,
\end{equation}
which implies of course \eqref{eq:88} in this case.  Fix a small
$\varepsilon>0$, e.g., $\varepsilon=\pi/8$. For each
$v\in\mathcal{E}(v_0)$ let $w_n:=\overline{u}_n\, v$ and define the set
$A_\varepsilon$ as in \eqref{eq:83}, with $\wtu=u_n$. Note that $|\nabla u_n(x)|=n$,
$x\in\Omega$, so by \rlemma{lem:ptwise} we have
\begin{equation}
\label{eq:58p}
|\nabla(u_n-v)|\geq n\sin\varepsilon \ \text{a.e. in }A_\varepsilon.
\end{equation}
Therefore,
\begin{equation}
\label{eq:72p}
   \int_{A_\varepsilon}|\nabla(u_n-v)|^p\geq |A_\varepsilon|(\sin\varepsilon)^pn^p=c_1|A_\varepsilon|n^p.
\end{equation}
Using
\eqref{eq:52} instead of \eqref{eq:45} in the computation leading to
\eqref{eq:58p} yields
\begin{equation}
  \label{eq:72}
 |\nabla(u_n-v)|\geq \left(\frac{\sin\varepsilon}{2}\right)(|\nabla
 u_n|+|\nabla v|)\geq \left(\frac{\sin\varepsilon}{2}\right)|\nabla
 w_n|,\ \text{ a.e. in }A_\varepsilon.
\end{equation}
We set $\widetilde w_n:=K_{\varepsilon}\circ w_n$ (see
\eqref{eq:Keps}) and recall
that $K_{\varepsilon}\in\text{Lip}(\so;\so)$, 
$\|\dt{K}_{\varepsilon}\|_\infty=\pi/(\pi-2\varepsilon)$ and
$\deg(K_{\varepsilon})=1$. We have $\widetilde
w_n\in\mathcal{E}(v_0)$ and $\nabla\widetilde w_n=0$ a.e.~in
$\Omega\setminus A_\varepsilon$. By \eqref{eq:1*},
\begin{equation}
\label{eq:76}
  \int_{A_\varepsilon}|\nabla \widetilde w_n|=\int_\Omega
  |\nabla\widetilde w_n|\geq \Sigma(v_0).
\end{equation}
Using H\"older inequality and \eqref{eq:76} gives
\begin{equation}
\label{eq:78}
  \int_{A_\varepsilon}|\nabla\widetilde w_n|^p\geq
\frac{\left(\int_{A_\varepsilon}|\nabla\widetilde w_n|\right)^p}{|A_\varepsilon|^{p-1}}\geq
\frac{(\Sigma(v_0))^p}{|A_\varepsilon|^{p-1}}.
\end{equation}
Since $|\nabla w_n|\geq (1-2\varepsilon/\pi)|\nabla\widetilde w_n|$ on
$\Omega$ we obtain by combining \eqref{eq:72} and \eqref{eq:78} that
\begin{equation}
  \label{eq:79}
  \int_{A_\varepsilon}|\nabla(u_n-v)|^p\geq   \frac{c_2}{|A_\varepsilon|^{p-1}},
\end{equation}
whence,
\begin{equation}
\label{eq:19}
  |A_\varepsilon|\geq c_2^{1/(p-1)}\left(\int_{A_\varepsilon}|\nabla(u_n-v)|^p\right)^{-1/(p-1)}.
\end{equation}
Plugging \eqref{eq:19} in \eqref{eq:72p} finally yields
 \begin{equation*}
  \int_{A_\varepsilon}|\nabla(u_n-v)|^p  \geq c_3n^{p-1},
 \end{equation*}
and \eqref{eq:50} follows.
\\[2mm]
{\it Proof of Theorem \rth{th:no-ub} in the general case.} Consider  an arbitrary
$u_0\in\wop(\Omega;\so)$. We set $u_n:=e^{\im
  nx_1}\, u_0\in\mathcal{E}(u_0)$.
By the triangle inequality,
\begin{equation*}
  |\nabla(e^{\im
    nx_1}-\overline{u}_0\, v)|=\left|\nabla\left(\overline{u}_0(u_n-v)\right)\right|\leq |\nabla(u_n-v)|+2|\nabla\overline{u}_0|.
\end{equation*}
Therefore,
\begin{equation*}
  \|\nabla(u_n-v)\|_{L^p(\Omega)}\geq \|\nabla(e^{\im
    nx_1}-\overline{u}_0v)\|_{L^p(\Omega)}-2\|\nabla u_0\|_{L^p(\Omega)},
\end{equation*}
and the result follows from the first part of the proof.
 \end{proof}
\subsection{An example of strict inequality in \eqref{eq:77}}
\label{subsec:strict}
\begin{proposition}
  \label{prop:strict}
There exist a smooth bounded simply connected
domain $\Omega$ in $\R^2$ and $u_0,v_0\in \bigcap_{1\leq p<2}\wop(\Omega;\so)$ such that 
\begin{equation}
    \label{eq:strict}
 \dist_{\wop}(\mathcal{E}(u_0),\mathcal{E}(v_0))>\left(\frac{2}{\pi}\right)\inf_{w\sim
  u_0{\overline v}_0 }\|\nabla w\|_{L^p(\Omega)},\  \fo 1<p<2.
  \end{equation}
\end{proposition}
\begin{proof}
 The construction resembles the one used in the proof of
 \cite[Proposition~4.1]{rs} (for a multiply connected domain and $p=2$),
 but the details of the proof are quite different. \\[2mm]
{\it Step 1.} Definition of $\Omega_\varepsilon$ and $u_0,v_0$\\
  Consider the three unit discs with centers at the points
  $a_{-}:=(-3,0),a:=(0,0)$ and $a_{+}:=(3,0)$, respectively: 
\bes
B_{-}:=B(a_{-},1),\ B:=B(a,1)\ \text{ and
}B_{+}:=B(a_{+},1).
\ees
For a small $\varepsilon\in(0,1/4)$, to be determined  later, define
the domain $\Omega_\varepsilon$ by
\begin{equation}
  \label{eq:Om-eps}
  \Omega_\varepsilon:=B_{-}\cup B\cup B_{+}\cup\{(x_1,x_2);\,
  x_1\in(-3,3), x_2\in(-\varepsilon,\varepsilon)\}.
\end{equation}

\begin{figure}[htbp] \label{fig:Omega}
 \begin{center}
 \scalebox{0.6}{\includegraphics{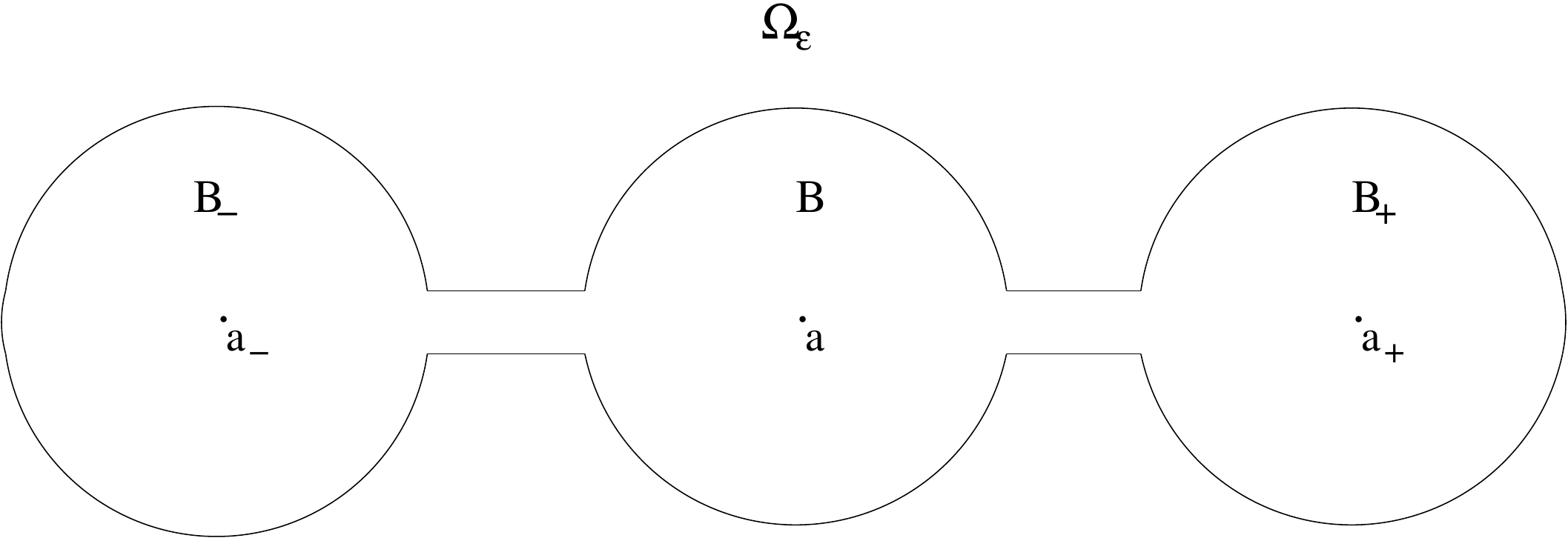}}
 \end{center}
 \caption{The domain $\Omega_\varepsilon$ (before smoothing)}
 \end{figure}
Hence, $B$ is connected to $B_{-}$ and $B_{+}$  by two narrow tubes (see
Figure~1). We can enlarge $\Omega_\varepsilon$ slightly near the
\enquote{corners}=the contact points of the tubes with the circles, in order
to have a smooth $\Omega_\varepsilon$. But we do it keeping the
following property:
\begin{equation}
  \label{eq:sym}
 \Omega_\varepsilon\text{ is symmetric with respect to reflections in both the
   $x$ and $y$-axis.}
\end{equation}
For later use we denote by
$\Omega_{+}$ and $\Omega_{-}$ the two components of
$\Omega_\varepsilon\setminus\overline B$ (with $\Omega_+\subset\{ z;\,
\textrm{Re }z>0\}$).
 We define the maps $u_0,v_0\in \bigcap_{1\leq p<2}\wop(\Omega_\varepsilon;\so)$
by
\begin{equation}
  \label{eq:u0v0}
   u_0:=\left(\frac{x-a_{-}}{|x-a_{-}|}\right)\left(\frac{x}{|x|}\right)^2\left(\frac{x-a_{+}}{|x-a_{+}|}\right)
,~v_0:=\left(\frac{x-a_{-}}{|x-a_{-}|}\right)\left(\frac{x}{|x|}\right)\left(\frac{x-a_{+}}{|x-a_{+}|}\right),
\end{equation}
and then
\begin{equation}
  \label{eq:w0}
 w_0:=u_0\, \overline{v}_0=\frac{x}{|x|}.
\end{equation}
{\it Step 2.} Properties of energy minimizers in $\mathcal{E}(w_0)$\\
We denote by $W_\varepsilon\in \wop(\Omega_\varepsilon;\so)$ a map realizing the
minimum in 
\begin{equation}
  \label{eq:33}
  S_\varepsilon=\inf_{w\sim w_0} \|\nabla w\|_{L^p(\Omega_\varepsilon)}.
\end{equation}
 Note that the minimizer $W_\varepsilon$ is unique, up to 
 multiplication by a complex constant of modulo one. This follows from the
 strict convexity of the functional:
 \begin{equation}
\label{eq:F-Funct}
   F(\varphi)=\int_{\Omega_\varepsilon}\left|\nabla\left(e^{\im\varphi}\frac{x}{|x|}\right)\right|^p \text{
     over }\varphi\in\wop(\Omega_\varepsilon;\R).
 \end{equation}
We next claim that
 \begin{equation}
   \label{eq:23}
   \int_{B} \left|\nabla\left(\frac{x}{|x|}\right)\right|^p\leq
   \int_{B} |\nabla W_\varepsilon|^p\leq\int_{\Omega_\varepsilon}|\nabla W_\varepsilon|^p\leq \int_{B} \left|\nabla\left(\frac{x}{|x|}\right)\right|^p+C\varepsilon^2,
 \end{equation}
 for some constant $C$. Here and in the sequel we denote by $C$
 different constants that are independent of $\varepsilon$ and $p$.
 Indeed, the first inequality in \eqref{eq:23} is clear since the restriction of 
 ${W_\varepsilon}$ to $B$ belongs to the class of $x/|x|$ in $B$,
 and the latter map is a minimizer of the energy in its class (see Remark~\ref{bn10} below). For
 the proof of the last inequality in \eqref{eq:23} it suffices to
 construct a comparison map ${\widetilde w}\in\mathcal{E}(w_0)$ as
 follows. We first set ${\widetilde w}=x/|x|$ in $B$. Then extend it to
 $\Omega_{+}\cap \{x_1\leq 1+\varepsilon\}$ in such a way that $\widetilde w\equiv
 \zeta$ (for some constant $\zeta\in\so$) on $\Omega_{+}\cap
 \{x_1=1+\varepsilon\}$. Such an extension can be constructed with
 $\|\nabla\widetilde w\|_{L^\infty}\leq C$, whence
 \begin{equation*}
   \int_{\Omega_{+}\cap \{x_1<1+\varepsilon\}}|\nabla{\widetilde w}|^p\leq C\varepsilon^2.
 \end{equation*}
In the remaining part of $\Omega_{+}$, namely $\Omega_{+}\cap \{x_1>
1+\varepsilon\}$ we simply set $\widetilde w\equiv \zeta$. We use
a similar construction for $\widetilde w$ on $\Omega_{-}$, and this
completes the proof of \eqref{eq:23}.

We shall also use a certain symmetry property of $W_\varepsilon$. We claim that:
\begin{equation}
  \label{eq:38}
  W_\varepsilon(x)=-W_\varepsilon(-x)~\text{ in }\Omega_\varepsilon.
\end{equation}
Indeed,
 since $W_\varepsilon(-x)$ is also a minimizer in \eqref{eq:33}, we must have
 \begin{equation}
  \label{eq:855}
   W_\varepsilon(-x)=e^{\im\alpha}\, W_\varepsilon(x)\ \text{for
     some constant }\alpha\in\R.
 \end{equation}
Write
\begin{equation}
  \label{eq:866}
W_\varepsilon=e^{\im\Psi_\varepsilon}\, \left(\frac{x}{|x|}\right),\
\text{with }\Psi_\varepsilon\in\wop(\Omega_\varepsilon;\R).
\end{equation}
 Plugging \eqref{eq:866} in \eqref{eq:855} gives
 \begin{equation*}
   -e^{\im\Psi_\varepsilon(-x)}\, \left(\frac{x}{|x|}\right)=e^{\im\alpha}\, e^{\im\Psi_\varepsilon(x)}\, \left(\frac{x}{|x|}\right),
 \end{equation*}
 whence
 $e^{\im(\Psi_\varepsilon(-x)-\Psi_\varepsilon(x))}=-e^{\im\alpha}$.
It follows that $\Psi_\varepsilon(-x)-\Psi_\varepsilon(x)\equiv const$
in $\Omega_\varepsilon$. Since  $\Psi_\varepsilon(-x)-\Psi_\varepsilon(x)$ is odd, 
it follows
that the constant must be zero. Hence
$\Psi_\varepsilon(-x)=\Psi_\varepsilon(x)$ a.e.~in
$\Omega_\varepsilon$, $e^{\im\alpha}=-1$ and \eqref{eq:38} follows
from \eqref{eq:855}.

 The main property of $W_\varepsilon$ that we need is the following: there exists
 $\zeta_\varepsilon\in\so$ such that
 \begin{align}
\label{eq:66}
 |W_\varepsilon-\zeta_\varepsilon|&\leq c_0\, \varepsilon^{2/p} \text{ on
 }B_{+},\\
\label{eq:71}
|W_\varepsilon+\zeta_\varepsilon|&\leq c_0\, \varepsilon^{2/p} \text{ on
 }B_{-}.
\end{align}
 In order to verify \eqref{eq:66}--\eqref{eq:71} we first notice that we may write
 $W_\varepsilon=e^{\im\Phi_\varepsilon}$ in
 $\Omega_\varepsilon\cap\{x_1>1\}$. Using \eqref{eq:23} and Fubini
 Theorem we can find $t_\varepsilon\in(1,3/2)$ such that the
 segment
 $I_\varepsilon=\{(t_\varepsilon,x_2);\,x_2\in(-\varepsilon,\varepsilon)\}$
 satisfies
 \begin{equation*}
   \int_{I_\varepsilon}|\nabla\Phi_\varepsilon|^p=\int_{I_\varepsilon}|\nabla
   W_\varepsilon|^p\leq C\varepsilon^2.
 \end{equation*}
 By H\"older inequality it follows that
 $|\Phi_\varepsilon(z_1)-\Phi_\varepsilon(z_2)|\leq
 C\varepsilon^{2/p}$ for all $z_1,z_2\in I_\varepsilon$. Hence, there
 exists $\alpha_\varepsilon\in\R$ satisfying
 \begin{equation}
   \label{eq:49}
 |\Phi_\varepsilon(z)-\alpha_\varepsilon|\leq
 C\varepsilon^{2/p},\ \forall\, z\in I_\varepsilon.
 \end{equation}
 We claim that \eqref{eq:49} continues to hold in
 $G_\varepsilon:=\Omega_\varepsilon\cap\{x_1>t_\varepsilon\}$, i.e., 
 \begin{equation}
   \label{eq:58}
|\Phi_\varepsilon(x)-\alpha_\varepsilon|\leq
 C\varepsilon^{2/p},\ \forall\, x\in G_\varepsilon.
 \end{equation}
 Indeed, defining
 \begin{equation*}
   {\widetilde\Phi}_\varepsilon(x):=\max\left(\alpha_\varepsilon-C\varepsilon^{2/p},\min(\Phi_\varepsilon(x),\alpha_\varepsilon+C\varepsilon^{2/p})\right),
 \end{equation*}
and then ${\widetilde W}_\varepsilon:=e^{\im\Phi_\varepsilon}$, we
clearly have $\int_{G_\varepsilon}|\nabla{\widetilde
  W}_\varepsilon|^p\leq\int_{G_\varepsilon}|\nabla W_\varepsilon|^p$,
with strict inequality, unless \eqref{eq:58} holds. Setting
$\zeta_\varepsilon:=e^{\im\alpha_\varepsilon}$, we deduce
\eqref{eq:66} from \eqref{eq:58}. Finally, using the symmetry
properties, \eqref{eq:sym} of $\Omega_\varepsilon$ and \eqref{eq:38}
of $\Psi_\varepsilon$, we easily deduce \eqref{eq:71} from
\eqref{eq:66}.\\[2mm]
{\it Step 3.} A basic estimate for maps in $\wop(\so;\so)$\\
The following claim provides a simple estimate which is essential for
the proof. The case
$p=2$ was proved in \cite[Lemma~4.1]{rs} and the generalization to any
$p\geq1$ is straightforward. We include the proof for the convenience
of the reader.
\begin{claim*}
 \label{claim:mindd}
 For any $p\geq1$, let  $f,g\in\wop(\so;\so)$ satisfy: 
\bes
\deg f=\deg g=k\neq 0\text{ and }|(f-g)(\zeta)|=\eta>0,
\ees
 for some point $\zeta\in\so$.
 Then,
\begin{equation}
\label{eq:duv}
\int_{\so}|\dot{f}-\dot{g}|^p\geq \frac{2\eta^p}{\pi^{p-1}}.
\end{equation}
\end{claim*}
\begin{proof}[Proof of Claim]
  Set $w:=f-g=w_1+\im w_2$. We may assume without loss of generality that $w(1)=(f-g)(1)=\eta\, \im$.
 Since $\deg(g)\neq 0$,  there exists a point $\theta_1\in(0,2\pi)$
 such that $g(e^{\im\theta_1})=\im$, whence
 $w_2(e^{\im\theta_1})=-t\im$ for some $t\geq0$.
H\"older's inequality, and a straightforward computation yield  
\bes
 \int_{\so} |w'|^p\geq \int_{\so}|w'_2|^p\geq \frac{(\eta+t)^p}{\theta_1^{p-1}}+\frac{(\eta+t)^p}{(2\pi-\theta_1)^{p-1}}
 \geq 2\frac{(\eta+t)^p}{\pi^{p-1}}>\frac{2\eta^p}{\pi^{p-1}},
\ees
 and \eqref{eq:duv} follows.
\end{proof}
\begin{remark}
  We thank an anonymous referee for suggesting a simplification of our
  original argument for the proof of the Claim, and for pointing out
  that it holds under the weaker assumption:
  either $f$ or $g$ has a nontrivial degree.
\end{remark}
\noindent
{\it Step 4.} Conclusion\\
Consider two sequences $\{u_n\}\subset\mathcal{E}(u_0)$ and
$\{v_n\}\subset\mathcal{E}(v_0)$ such that
\begin{equation}
\label{eq:75}
  \lim_{n\to\infty}\|\nabla(u_n-v_n)\|_{L^p(\Omega_\varepsilon)}=\dist_{\wop}(\mathcal{E}(u_0),\mathcal{E}(v_0)).
\end{equation}
By a standard density argument we may assume that
$u_n,v_n\in
C^\infty(\overline{\Omega}_\varepsilon\setminus\{a_{-},a,a_{+}\})$ for
all $n$.
 Assume by contradiction that
 \begin{equation}
   \label{eq:48}
  \dist_{\wop}(\mathcal{E}(u_0),\mathcal{E}(v_0))
  =\left(\frac{2}{\pi}\right)S_\varepsilon=\left(\frac{2}{\pi}\right)\min_{w\sim
    w_0} \|\nabla w\|_{L^p(\Omega_\varepsilon)}~\text{(see \eqref{eq:33})}.
 \end{equation}
By \eqref{eq:48} and \eqref{eq:23} there exists a constant $C_0$ such
that
\begin{equation}
  \label{eq:for-cont}
  \int_{\Omega_\varepsilon}|\nabla(u_n-v_n)|^p\leq \left(\frac{2}{\pi}\right)^p\int_{B}
  \left|\nabla\left(\frac{x}{|x|}\right)\right|^p+C_0\varepsilon^2,\
  \forall\,  n\geq n_0(\varepsilon).
\end{equation}

Set $w_n:=u_n\, {\overline v_n}$ and note that by the same computation as
in \eqref{eq:90}--\eqref{eq:89} we have
\begin{equation}
\label{eq:80}
  \int_{\Omega_\varepsilon}|\nabla(u_n-v_n)|^p\geq \int_{\Omega_\varepsilon}|\nabla(|u_n-v_n|)|^p=\left(\frac{2}{\pi}\right)^p
\int_{\Omega_\varepsilon}|\nabla(T\circ w_n)|^p\geq \left(\frac{2}{\pi}\right)^pS^p_\varepsilon
\end{equation}
 (recall that $T$ is defined in \eqref{eq:41}). Combining
 \eqref{eq:75},\eqref{eq:48} and \eqref{eq:80} yields that
 $\wtw_n:=T\circ w_n$ satisfies
 \begin{equation*}
   \lim_{n\to\infty}\|\nabla \wtw_n\|_{L^p(\Omega_\varepsilon)}=S_\varepsilon,
 \end{equation*}
and up to passing to a subsequence we have
\begin{equation}
  \label{eq:84}
W_\varepsilon=\lim_{n\to\infty} \wtw_n~\text{ in }\wop(\Omega;\so),
\end{equation}
 where $W_\varepsilon$ is a minimizer in \eqref{eq:33}. 
 Recall that $W_\varepsilon$ is unique up to rotations; the particular
 $W_\varepsilon$ in \eqref{eq:84} is chosen by the subsequence. For
 any $\zeta\in\so$ we have
 $\max(|\zeta-1|,|\zeta+1|)\geq\sqrt{2}$. In
 particular, for
 $\zeta_\varepsilon$ associated with $W_\varepsilon$ (see
 \eqref{eq:66}--\eqref{eq:71}) 
 we may assume without loss of generality that
 \begin{equation}
     \label{eq:85}
 |\zeta_\varepsilon-1|\geq\sqrt{2}.
 \end{equation}
 By  \eqref{eq:84} and Egorov Theorem there exists
 $A_\varepsilon\subset\Omega_\varepsilon$ satisfying
 \begin{equation}
   \label{eq:Egorov}
 |A_\varepsilon|\leq\varepsilon~\text{ and }\wtw_n\to W_\varepsilon\text{
   uniformly on }\Omega_\varepsilon\setminus A_\varepsilon, 
 \end{equation}
 again, after passing to a subsequence.
Combining \eqref{eq:Egorov} with \eqref{eq:85} and \eqref{eq:66} yields
\begin{equation*}
 |\wtw_n-1|\geq \sqrt{2}-2c_0\, \varepsilon\ \text{ on
 }B_{+}\setminus A_\varepsilon,\ \fo n\geq n_1(\varepsilon),
\end{equation*}
and choosing $\varepsilon<(\sqrt{2}-1)/(2c_0)$ guarantees that
\begin{equation}
  \label{eq:86}
 |\wtw_n-1|\geq 1\ \text{ on
 }B_{+}\setminus A_\varepsilon,\ \fo n\geq n_1(\varepsilon).
\end{equation}
Going back to the definition of $T$ in \eqref{eq:41}, we find by a
simple computation the following equivalences for
$e^{\im\theta}=T(e^{\im\varphi})$ (with $\theta\in(-\pi,\pi)$):
\begin{equation}
  \label{eq:93}
 |T(e^{\im\varphi})-1|\geq
 1~\Longleftrightarrow~|\theta|=\pi|\sin(\varphi/2)|\geq\pi/3~\Longleftrightarrow~|e^{\im\varphi}-1|=2|\sin(\varphi/2)|\geq 2/3.
\end{equation}
Using \eqref{eq:93} we may rewrite \eqref{eq:86} in terms of the
original sequence $\{w_n\}$:
\begin{equation}
  \label{eq:94}
 |u_n-v_n|= |w_n-1|\geq 2/3~\text{ on
 }B_{+}\setminus A_\varepsilon,\ \fo  n\geq n_1(\varepsilon).
\end{equation}
Consider the set
\begin{equation}
\label{eq:91}
  \Lambda_\varepsilon=\{r\in(1/2,1);\,\partial B(a_{+},r)\subset A_\varepsilon\}.
\end{equation}
By \eqref{eq:Egorov} we clearly have $\varepsilon\geq
|A_\varepsilon|\geq (1/2)|\Lambda_\varepsilon|\cdot2\pi$, whence
\begin{equation}
  \label{eq:92}
 |\Lambda_\varepsilon|\leq \frac{\varepsilon}{\pi}.
\end{equation}
For $n\geq n_1(\varepsilon)$ we have: on each circle $\partial B(a_{+},r)$ with $r\in
(1/2,1)\setminus\Lambda_\varepsilon$ there exists at least one point
where $|u_n-v_n|\geq2/3$. Thus we may apply the Claim from Step 3 with $\eta:=2/3$, $f:=u_n\big|_{\partial B(a_{+},r)}$ and
$g:=v_n\big|_{\partial B(a_{+},r)}$ to obtain by \eqref{eq:duv} (after a suitable rescaling):
\begin{equation}
  \label{eq:95}
 \int_{\partial B(a_{+},r)}|\nabla(u_n-v_n)|^p\geq 2(r\pi)^{1-p}\left(\frac{2}{3}\right)^p\geq(2\pi)\left(\frac{2}{3\pi}\right)^p:=\gamma_p.
\end{equation}
Integrating \eqref{eq:95}, taking into account \eqref{eq:92},  yields
\begin{equation}
  \label{eq:96}
 \int_{B_{+}}|\nabla(u_n-v_n)|^p\geq
 \int_{(1/2,1)\setminus\Lambda_\varepsilon} \int_{\partial
   B(a_{+},r)}|\nabla(u_n-v_n)|^p\geq (1/2-\varepsilon/\pi)\gamma_p.
\end{equation}
In addition, by \eqref{eq:77}, applied to $u_n\big|_B, v_n\big|_B$, we clearly have
\bes\int_{B}|\nabla(u_n-v_n)|^p\geq\left(\frac{2}{\pi}\right)^p\int_{B}\left|\nabla\left(\frac{x}{|x|}\right)\right|^p,\ees
which together with \eqref{eq:96} gives 
\begin{equation}
  \label{eq:97}
 \int_{\Omega_\varepsilon}|\nabla(u_n-v_n)|^p\geq\left(\frac{2}{\pi}\right)^p\int_{B}\left|\nabla\left(\frac{x}{|x|}\right)\right|^p+
 (1/2-\varepsilon/\pi)\gamma_p,\  \fo  n\geq n_1(\varepsilon).
\end{equation}
The inequality \eqref{eq:97} clearly contradicts \eqref{eq:for-cont}
for $n$ large enough
if $\varepsilon$ is chosen sufficiently small.
\end{proof}
\begin{remark}
\label{bn10}
In the course of the proof of \rprop{prop:strict} we used the following fact:\\[2mm]
Let $1\le p<2$ and let $\Omega$ be the unit disc. Set $u_0(x):=x/|x|$, $\fo x\in\Omega$. Then 
\be
\label{bn11}
\int_\Omega |\na u|^p\ge \int_\Omega |\na u_0|^p ,\ \fo u\in{\cal{E}}(u_0).
\ee
We sketch the proof for the convenience of the reader.\\[2mm]
Let $u\in{\cal{E}}(u_0)$. Let $C_r:=\{ z;\, |z|=r\}$. Since we may write $u=e^{\im\va}\, u_0$, with $\va\in W^{1,p}$, for a.e. $r\in (0,1)$ we have $u\big|_{C_r}\in W^{1,p}(C_r;\so)$ and $\d\deg \left(u\big|_{C_r}\right)=1$. This implies that for a.e. $r\in (0,1)$ we have
\be
\label{bn12}
\int_{C_r}|\na u|\ge \int_{C_r}\left|\dot u\right|=\int_{C_r}\left|u\wedge\dot u\right|\ge \int_{C_r} u\wedge\dot u=2\pi=\int_{C_r}u_0\wedge\dot {u}_0=\int_{C_r}|\na u_0|.
\ee
In case $p=1$ integration over $r\in(0,1)$ of \eqref{bn12} yields
\eqref{bn11}. In case $1<p<2$ we use
\eqref{bn12} and H\"older inequality, and then integration over $r$ yields
\bes
\int_\Omega |\na u|^p\ge
2\pi\int_0^1\frac{dr}{r^{p-1}}=\frac{2\pi}{2-p}=\int_\Omega |\na
u_0|^p, \text{ and \eqref{bn11} follows.}
\ees
\\[2mm]
Examining  the equality cases for the inequalities in \eqref{bn12}
(and in H\"older inequality when $1<p<2$) we obtain in addition the
following conclusion: equality holds in \eqref{bn11} if and only if\\[2mm]
(i) for $1<p<2$, $u=e^{\im\alpha}u_0$  for some constant $\alpha$;\\
(ii) for $p=1$, $u(re^{\im\theta})=e^{i\varphi(\theta)}$ where
$\varphi\in W^{1,1}([0,2\pi];\R)$ satisfies
$\varphi(2\pi)-\varphi(0)=2\pi$ and $\varphi'\geq0$ a.e.~on
$[0,2\pi]$.\\
The difference between  (i) and (ii) is the main
reason why for the same $u_0$ and $v_0$ as in the proof of
\rprop{prop:strict}, we have  the strict inequality 
\eqref{eq:strict} for $1<p<2$, while for $p=1$  the equality \eqref{eq:19a} holds.
\end{remark}
\begin{remark}
 Consider the maps $u_1:=x/|x|$ and $v_1:=1$ in $\Omega_\varepsilon$ (as in the proof of
 \rprop{prop:strict}).
 By \rth{th:smooth} we have
 $\dist_{\wop}(\mathcal{E}(u_1),\mathcal{E}(v_1))=(2/\pi)\inf_{w\sim
   u_1{\overline v}_1}\|\nabla
 w\|_{L^p(\Omega_\varepsilon)}$. Therefore, we have
 \begin{equation*}
   \dist_{\wop}(\mathcal{E}(u_1),\mathcal{E}(v_1))\neq \dist_{\wop}(\mathcal{E}(u_0),\mathcal{E}(v_0))
 \end{equation*}
 although $u_1{\overline v}_1=u_0{\overline v}_0$. This shows that in
 general it is not even true that
 $\dist_{\wop}(\mathcal{E}(u),\mathcal{E}(v))$ depends only on
 $\mathcal{E}(u{\overline v})$ when $1<p<2$. A similar phenomenon
 occurs when $\Omega$ is multiply connected and $p=2$ (see \cite[Remark~4.1]{rs});
  a comparable argument works for $p>2$.
\end{remark}

\section*{Appendix. Proof of \rprop{prop:dipole}}
 \setcounter{equation}{0}
 \renewcommand{\theequation}{\Alph{section}.\arabic{equation}}
  \setcounter{section}{1}
  \renewcommand{\thelemma}{A}
 \begin{proof}[Proof of \rprop{prop:dipole}]
   We fix a sequence $\varepsilon_n\searrow0$ and use \eqref{eq:1*} to
   find a sequence $\{v_n\}\subset\mathcal{E}(u)$ such that
   \begin{equation}
     \label{eq:109}
     \int_\Omega |\nabla v_n|\leq
     \Sigma(u)+\varepsilon_n,\ \forall\, n.
   \end{equation}
For $\theta\in[0,2\pi)$ define $\Psi_{n,\theta}\in\text{Lip}(\so;\so)$ by
\begin{equation}
\label{eq:110}
 \Psi_{n,\theta}(z):=\begin{cases}
 e^{\im\pi (1+2(\varphi-\theta)/\varepsilon_n)},&\text{if } z=e^{\im\varphi}\in \mathcal{A}(\theta-\varepsilon_n/2,\theta+\varepsilon_n/2)\\
1, &\text{if } z\notin \mathcal{A}(\theta-\varepsilon_n/2,\theta+\varepsilon_n/2)
\end{cases}.
\end{equation}
 Clearly $\deg \Psi_{n,\theta}=1$, so setting
 $w_{n,\theta}:=\Psi_{n,\theta}\circ v_n$, we have, by \rlemma{lem:T},
 $w_{n,\theta}\sim v_n\sim u$. Moreover,
 \begin{equation}
\label{eq:111}
  |\nabla w_{n,\theta}(x)|=\begin{cases}
(2\pi/\varepsilon_n)\, |\nabla v_n(x)|,&\text{if } v_n(x)\in\mathcal{A}(\theta-\varepsilon_n/2,\theta+\varepsilon_n/2)\\
0,  &\text{if } v_n(x)\notin\mathcal{A}(\theta-\varepsilon_n/2,\theta+\varepsilon_n/2)
\end{cases}.
\end{equation}

Set $A_{n}(x):=\{\theta\in[0,2\pi);
  v_n(x)\in\mathcal{A}(\theta-\varepsilon_n/2,\theta+\varepsilon_n/2)\}$.
    We have
\begin{equation}
\label{eq:112}
  \int_0^{2\pi} \int_\Omega |\nabla
  w_{n,\theta}|\,dx\,d\theta=\frac{2\pi}{\varepsilon_n}\int_\Omega|\nabla
  v_n(x)|\,|A_{n}(x)|\,d
x=2\pi\int_\Omega |\nabla v_n|,
\end{equation}
and
\begin{equation}
  \label{eq:113}
\int_0^{2\pi} |\{w_{n,\theta}\neq 1\}|\,d\theta=\int_\Omega |A_n(x)|\,dx=\varepsilon_n\, |\Omega|.
\end{equation}
Combining \eqref{eq:109} with \eqref{eq:112}--\eqref{eq:113} yields
\begin{equation}
  \label{eq:114}
\int_0^{2\pi} \left(|\{w_{n,\theta}\neq
1\}|/{\varepsilon}^{1/2}_n+\int_\Omega|\nabla w_{n,\theta} |\right)\,d\theta\leq |\Omega|\, {\varepsilon}^{1/2}_n+2\pi(\Sigma(u)+\varepsilon_n). 
\end{equation}
From \eqref{eq:114} it follows that there exists $\theta_n\in[0,2\pi)$
such that
\begin{equation*}
  |\{w_{n,\theta_n}\neq
1\}|/{\varepsilon}^{1/2}_n+\int_\Omega|\nabla w_{n,\theta_n} |\leq |\Omega|\, {\varepsilon}^{1/2}_n/(2\pi)+\Sigma(u)+\varepsilon_n;
\end{equation*}
 so clearly a subsequence of $u_n:=w_{n,\theta_n}$ satisfies \eqref{eq:dipole}.

 \end{proof}


\begin{thebibliography}{6}
\bibitem{abl} F.~Almgren, W.~Browder and E.~H.~Lieb, {\em Co-area,
    liquid crystals, and minimal surfaces}, Partial differential
  equations (Tianjin, 1986), 1--22, Lecture Notes in Math., {\bf 1306}, Springer, Berlin, 1988.
 \bibitem{ball} J.~Ball, {\em  Convexity conditions and
     existence theorems in nonlinear elasticity}, 
Arch. Rational Mech. Anal. {\bf 63} (1976/77),  337--403.


  \bibitem{bethuel-IAHP}  F.~Bethuel, {\em A characterization of maps
     in $H^1(B^3,{\mathbb S}^2)$ which can be approximated by smooth
     maps}, Ann. Inst. H. Poincar\'e Anal. Non Lin\'eaire {\bf 7} (1990), 269--286. 
 \bibitem{bethuelzheng}
F.~Bethuel and X.M.~Zheng, 
{\em Density of smooth functions between two manifolds in {S}obolev
  spaces}, 
 J. Funct. Anal., {\bf 80} (1988), 60--75.
\bibitem{br1987} H. Brezis,  {\em Liquid crystals and energy estimates
    for $S^2$-valued maps}, Theory and applications of liquid crystals
  (Minneapolis, Minn., 1985), 31--52, IMA Vol. Math. Appl., {\bf 5}, Springer, New York, 1987.
 \bibitem{br-FA} H.~Brezis, Functional analysis, Sobolev spaces and partial differential equations, Universitext. Springer, New York, 2011.
 \bibitem{bcl}  
 H.~Brezis, J.M.~Coron and E.~Lieb,  {\em Harmonic maps with defects}, Comm. Math. Phys. {\bf 107} (1986), 649--705.
\bibitem{bbm-lifting} 
J.~Bourgain, H.~Brezis and P.~Mironescu, {\em
    Lifting in Sobolev spaces},  J. Anal. Math. {\bf 80} (2000), 37--86.

\bibitem{b-li} H.~Brezis and Y.Y.~Li,  {\em Topology and Sobolev
    spaces}, J. Funct. Anal. {\bf 183} (2001), 321--369.
 \bibitem{bm-RACSAM} H.~Brezis and P.~Mironescu, {\em On some questions
     of topology for $\so$-valued fractional Sobolev spaces},
   RACSAM. Rev. R. Acad. Cienc. Exactas Fís. Nat. Ser. A Mat. {\bf 95} (2001), 121--143.

\bibitem{bm} H.~Brezis and P.~Mironescu, Sobolev maps with values into
  the circle, Birkh\"auser (in preparation).
\bibitem{bmp} H.~Brezis, P.~Mironescu and A. Ponce, {\em
    $W^{1,1}$-maps with values into ${\mathbb S}^1$}, in Geometric analysis of PDE and several complex
  variables, Contemp. Math., {\bf 368}, Amer. Math. Soc.,
  Providence, RI, 2005, 69--100.
\bibitem{bms2} H.~Brezis, P.~Mironescu and I.~Shafrir, {\em Distances
    between homotopy classes of  $W^{s,p}({\mathbb S}^N;{\mathbb
      S}^N)$},  ESAIM Control Optim. Calc. Var. {\bf 22} (2016), 1204--1235.
\bibitem{bmsCRAS} H.~Brezis, P.~Mironescu and I.~Shafrir, {\em
    Distances between classes of sphere-valued Sobolev maps},
  C. R. Math. Acad. Sci. Paris {\bf 354} (2016), 677--684. 
\bibitem{bms17} H.~Brezis, P.~Mironescu and I.~Shafrir, in preparation.
\bibitem{BN}  H.~Brezis and L.~Nirenberg, {\em Degree theory and
    BMO. I. Compact manifolds without boundaries}, Selecta
  Math. (N.S.) {\bf 1} (1995), 197--263.
\bibitem{carbou}  G.~Carbou, {\em Applications harmoniques \`a valeurs
    dans un cercle}, C. R. Acad. Sci. Paris S\'er. I Math. {\bf 314}
    (1992), 359--362.
\bibitem{DI} J.~D\'avila and R.~Ignat,
 {\em Lifting of {BV} functions with values in {$S^1$}},
  C. R. Math. Acad. Sci. Paris {\bf 337} (2003), 159--164.
\bibitem{demengel}  F.~Demengel, {\em Une caract\'erisation des
    applications de $W^{1,p}(B^N,\so)$ qui peuvent \^etre approch\'ees
    par des fonctions r\'eguli\`eres}, C. R. Acad. Sci. Paris S\'er. I
  Math. {\bf 310} (1990), 553--557.
\bibitem{evans} L.C.~Evans, {\em Partial differential equations and Monge-Kantorovich mass transfer}, Current developments in mathematics, 1997 (Cambridge, MA), 65--126, Int. Press, Boston, MA, 1999. 
\bibitem{GMS} M.~Giaquinta, G.~Modica and J.~Sou\v cek, 
{\em Cartesian currents in the calculus of variations. II. 
Variational integrals}, Ergebnisse der Mathematik und ihrer Grenzgebiete. 3. Folge. A Series of Modern Surveys in Mathematics [Results in Mathematics and Related Areas. 3rd Series. A Series of Modern Surveys in Mathematics], 38. Springer-Verlag, Berlin, 1998.
\bibitem{hang-lin}  F.~Hang and F.H.~Lin, {\em Topology of Sobolev
    mappings II}, Acta Math. {\bf 191} (2003),  55--107. 
\bibitem{ls} S.~Levi and I.~Shafrir, {\em On the distance between homotopy classes of maps between spheres},
J. Fixed Point Theory Appl. {\bf 15} (2014),  501--518. 
\bibitem{M} B.~Merlet, 
   {\em Two remarks on liftings of maps with values into $S^1$},
  C. R. Math. Acad. Sci. Paris {\bf 343} (2006), 467--472.
\bibitem{pol}  A.~Poliakovsky, {\em On a minimization problem related
    to lifting of BV functions with values in $S^1$},
  C. R. Math. Acad. Sci. Paris {\bf 339} (2004),855--860.
\bibitem{P-VS} A.C.~Ponce and J.~Van Schaftingen, {\em Closure of
    smooth maps in 
$W^{1,p}(B^3;{\mathbb S}^2)$}, Differential
Integral Equations, {\bf 22} (2009), 881--900.
\bibitem{rs} J.~Rubinstein and I.~Shafrir, 
{\em The distance between homotopy classes of ${\mathbb S}^1$-valued maps in multiply connected domains},
Israel J. Math. {\bf 160} (2007), 41--59. 
\bibitem{rub-ster} J.~Rubinstein and  P.~Sternberg, 
{\em Homotopy classification of minimizers of the Ginzburg-Landau energy and the existence of permanent currents},
Comm. Math. Phys. {\bf 179} (1996),  257--263. 
\bibitem{santam}F.~Santambrogio, Optimal transport for applied mathematicians, Calculus of variations, PDEs, and modeling. Progress in Nonlinear Differential Equations and their Applications, 87. Birkhäuser/Springer, Cham, 2015.
\bibitem{su} R.~Schoen and  K.~Uhlenbeck, {\em Boundary regularity and
    the Dirichlet problem for harmonic maps}, J. Differential
  Geom. {\bf 18} (1983), 253--268. 
\bibitem{villani} C.~Villani, Optimal transport. Old and new, Grundlehren der Mathematischen Wissenschaften 338. Springer-Verlag, Berlin, 2009. 
\bibitem{white} B.~White, {\em Homotopy classes in Sobolev spaces and
    the existence of energy minimizing maps}, Acta Math. {\bf 160} (1988),1--17.
\end{thebibliography}
\end{document}